\documentclass[12pt]{amsart}
\usepackage{amsmath}
\usepackage{mathdots}
\usepackage{amssymb}
\usepackage{amscd}
\usepackage[matrix,arrow]{xy}
\usepackage{amsmath,amsthm,amsfonts}
\usepackage{color}
\usepackage{bbm}
\usepackage{tikz}
\usepackage{tikz-cd}
\usepackage{textcomp}
\usepackage{geometry}
\theoremstyle{plain}
\newtheorem{theorem}{Theorem}[section]
\newtheorem{lemma}[theorem]{Lemma}
\newtheorem{proposition}[theorem]{Proposition}

\newtheorem{corollary}[theorem]{Corollary}

\numberwithin{equation}{section}

\theoremstyle{remark}
\newtheorem{example}[theorem]{Example}
\newtheorem{remark}[theorem]{Remark}

\theoremstyle{definition}
\newtheorem{definition}[theorem]{Definition}

\def \colim {\operatorname{colim}}
\def \Ker {\operatorname{Ker}}
\def \Coker {\operatorname{Coker}}
\def \SNilp {\operatorname{SNil}}

\newcommand{\B}{{\rm B}}

\newcommand{\Ab}{{\mathcal A}b}
\newcommand{\Cc}{\mathcal C}

\newcommand{\Z}{\mathbb Z}
\newcommand{\bF}{\mathbb F}
\newcommand{\LL}{\mathbb L}
\newcommand{\Q}{\mathbb Q}

\title[Relative Milnor $K$-groups and differential forms]{Relative Milnor $K$-groups and differential forms of split nilpotent extensions}

\author{Sergey Gorchinskiy}
\address{Steklov Mathematical Institute of Russian Academy of Sciences}
\address{National Research University Higher School of Economics, Russian Federation}
\email{gorchins@mi.ras.ru}

\author{Dimitrii Tyurin}
\address{National Research University Higher School of Economics, Russian Federation}
\email{dimtyurin@mail.ru}

\date{}

\begin{document}

\maketitle
\begin{abstract}
Let $R$ be a commutative ring and $I\subset R$ be a nilpotent ideal such that the quotient~$R/I$ splits out of $R$ as a ring. Let $N$ be a natural number such that~${I^N=0}$. We establish a canonical isomorphism between the relative Milnor $K$-group $K^{M}_{n+1}(R,I)$ and the quotient of the relative module of differential forms $\Omega^n_{R,I}/d\,\Omega^{n-1}_{R,I}$ assuming that~$N!$ is invertible in~$R$ and that the ring~$R$ is weakly $5$-fold stable. The latter means that any \mbox{$4$-tuple} of elements in~$R$ can be shifted by an invertible element to become a $4$-tuple of invertible elements.
\end{abstract}

%\tableofcontents

\section{Introduction}

An important invariant of a commutative associative unital ring $R$ is its Milnor \mbox{$K$-group}~$K^{M}_{n}(R)$, where $n\geqslant 0$ is a natural number. Recall that this is the degree~$n$ component in the
quotient of the tensor ring~$(R^{*})^{\otimes\bullet}$ of the group $R^*$ of invertible elements in $R$ over the two-sided ideal generated by elements of type~$r\otimes(1-r)$, called Steinberg relations, where ${r,1-r\in R^*}$.
Milnor \mbox{$K$-groups} play a fundamental role in various domains of algebra and arithmetics, including class field theory. However, it is rather hard to compute Milnor $K$-groups in general as their definition involves a delicate interplay between the additive and multiplicative structures in the ring.

At the same time, one has another invariant of $R$, its module of (absolute) differential forms~$\Omega^n_R$ and it is relatively easy to calculate it explicitly. There is a canonical homomorphism of groups ${d\log\colon K^{M}_{n+1}(R)\to\Omega^{n+1}_R}$, where $n\geqslant 0$. Though the map~$d\log$ is very far from being an isomorphism in general, it becomes much closer to an isomorphism after one passes to a relative version.

\medskip

Namely, let $I\subset R$ be a nilpotent ideal such that the quotient map~${R\to R/I}$ has a section by a ring homomorphism, that is, let $(R,I)$ be a split nilpotent extension of the quotient~$R/I$ (see Definition~\ref{def:splitnilp}). The corresponding relative Milnor $K$-group~${K^{M}_{n}(R,I)}$ is defined as the kernel of the (surjective) homomorphism~${K^{M}_{n}(R)\to K^{M}_{n}(R/I)}$ (see Definition~\ref{def:relmil}), the module of relative differential forms~$\Omega^{n}_{R,I}$ is defined similarly (see Definition~\ref{def:relforms}), and one has the relative~$d\log$ map
$$
d\log:K^{M}_{n+1}(R,I)\longrightarrow\Omega^{n+1}_{R,I}\,,\qquad n\geqslant 0\,.
$$
Suppose also that $N!$ is invertible in $R$, where $N\geqslant 1$ is a natural number such that~${I^N=0}$. Then the relative $d\log$ map can be integrated canonically to a map
$$
{\rm B}\;:\;K^{M}_{n+1}(R,I)\longrightarrow\Omega^{n}_{R,I}/d\,\Omega^{n-1}_{R,I}
$$
(see Definition~\ref{def:Blochmap} and~Theorem~\ref{the:Bloch}). This was first observed by Bloch~\cite[\S\,1]{Bloch75} in a particular case and we call~$\B$ a Bloch map.

\medskip

Our main result (see Theorem~\ref{thm:main}) is that the Bloch map is an isomorphism if the ring $R$ is, in addition, weakly $5$-fold stable in the sense of Morrow~\cite[Def.\,3.1]{Morrow} (see Definition~\ref{def:weakst}). The latter condition means that for any 4-tuple~${r_1,\ldots,r_4\in R}$, there is an invertible element $r\in R^*$ such that the elements ${r_1+r,\ldots,r_4+r}$ are invertible in~$R$.

This statement can be considered as a version of the famous Goodwillie's theorem~\cite{Goodwillie} with algebraic $K$-groups replaced by Milnor $K$-groups.

\medskip

Let us compare our theorem with previously known results in this direction. Originally, van der Kallen~\cite{vdKallen71} proved an isomorphism
$$
K_2\big(S[\varepsilon],(\varepsilon)\big)\stackrel{\sim}\longrightarrow \Omega^1_S\,,
$$
where~$K_2\big(S[\varepsilon],(\varepsilon)\big)$ is the relative algebraic \mbox{$K$-group}, $\varepsilon$ is a formal variable that satisfies $\varepsilon^2=0$, and $S$ is a ring such that $2$ is invertible in it. Note that there is an isomorphism ${\Omega^1_S\simeq \Omega^1_{S[\varepsilon],(\varepsilon)}/d\,(\varepsilon)}$ (see Example~\ref{examp:Blochmap2}). The case of algebraic $K$-groups of arbitrary degree $n\geqslant 0$ was investigated then by Bloch~\cite{Bloch73}. Besides, Bloch~\cite[Theor.\,0.1]{Bloch75} proved an isomorphism
$$
K_2(R,I)\stackrel{\sim}\longrightarrow \Omega^1_{R,I}/d\,I
$$
assuming that $(R,I)$ is a split nilpotent extension of a local $\mathbb{Q}$-algebra $R/I$. This was generalized by Maazen and Stienstra~\cite[\S\,3.12]{MaazenStienstra} who found another proof and also replaced the condition in Bloch's theorem that $R/I$ is a local $\mathbb{Q}$-algebra with $R/I$ being any ring such that~$N!$ is invertible in it.

On the other hand, it follows from a result of van der Kallen~\cite[Cor.\,8.5]{vdKallen} that there is an isomorphism ${K^{M}_{2}(R,I)\simeq K_{2}(R,I)}$ when the ring~$R$ is $5$-fold stable. In particular, this holds when $R$ is a semi-local ring such that its residue fields are not isomorphic to~$\bF_2$, $\bF_3$, $\bF_4$, $\bF_5$. Note that being a $5$-fold stable ring is a stronger condition than being a weakly $5$-fold stable ring. For example, the ring of Laurent series over a ring is typically not $5$-fold stable, while it is always weakly $5$-fold stable, see~\cite[Rem.\,3.5]{Morrow} or~\cite[Rem.\,2.4(ii)]{GorchinskiyOsipov2015Miln}.

Combining all these results, one obtains that~${\rm B}$ is an isomorphism when $n=1$, the natural number $N!$ is invertible in $R$, and $R$ is $5$-fold stable. Finally, Dribus~\cite{Dribus} deduced from this statement the case of arbitrary degree~$n\geqslant 0$.

Note that this approach is based on difficult inexplicit theorems from algebraic \mbox{$K$-theory}, while both sides of the isomorphism are given by explicit generators and relations. Thus it is natural to find a more direct proof which would be basically reduced just to Steinberg relations. We give such a proof in our paper. Moreover, unlike the above approach, our result is valid for weakly $5$-fold stable rings, not only for $5$-fold stable rings.

Previously, Gorchinskiy and \mbox{Osipov}~\cite[Theor.\,2.9]{GorchinskiyOsipov2015Miln} proved that $\rm B$ is an isomorphism in the case $R=S[\varepsilon]$, $I=(\varepsilon)$, where $\varepsilon^2=0$ as above and $S$ is a weakly $5$-fold stable ring such that $2$ is invertible in it (see also Example~\ref{examp:Blochmap2}). They applied this result to the study of the higher-dimensional Contou-Carr\`ere symbol in the series of papers~\cite{GorchinskiyOsipov15UMN},~\cite{GorchinskiyOsipov15MS},~\cite{GorchinskiyOsipov16Tr},~\cite{GorchinskiyOsipov16FFA}. The approach in~\cite{GorchinskiyOsipov2015Miln} is based on an explicit analysis of elements in Milnor $K$-groups. We reduce our theorem to this case using that relative Milnor $K$-groups and modules of differential forms commute with a certain class of not filtered colimits and also applying several new tricks to deal with elements in Milnor $K$-groups.

\medskip

The paper is organized as follows. In Section~\ref{sec:prelim}, we introduce our main objects of study and state the main results. Subsection~\ref{subsec:Milnorbasic} and Subsection~\ref{subsec:basicdiffforms} consist of recollections on Milnor $K$-groups and modules of differential forms, respectively, including their relative versions $K_n^M(R,I)$ and $\Omega^n_{R,I}$ for a (nilpotent) ideal $I\subset R$. In particular, we give an explicit description of generators for the relative groups (see Lemma~\ref{lem:generMilnorK} and Lemma~\ref{lem:generKahler}), which is used frequently in what follows. In Subsection~\ref{subsec:Blochmap}, we define the Bloch map ${\B\colon K_{n+1}^M(R,I)\to\Omega^n_{R,I}/d\,\Omega^{n-1}_{R,I}}$ (see Definition~\ref{def:Blochmap}) and mention some of its general properties (see Lemma~\ref{lem:surjBlochmap} and Remark~\ref{rmk:Blochmap}). Subsection~\ref{subsec:mainresults} contains the formulations of our main results on the Bloch map, namely, its existence (see Theorem~\ref{the:Bloch}) and the property of being an isomorphism (see Theorem~\ref{thm:main}). We have also included a short discussion of possible generalizations to a non-split case (see Remark~\ref{rem:nonsplit}).

Section~\ref{sec:Bloch} is devoted to a proof of the existence of the Bloch map. First, in Subsection~\ref{subsec:splitnilp}, we introduce the category $\SNilp_N(S)$ of split nilpotent extensions of a given ring $S$ of nilpotency degree $N$ (see Definition~\ref{def:splitnilp}). Also, we define finite free objects $(R_{N,m},I_{N,m})$ in this category (see Definition~\ref{def:elem}), which are the quotients~${R_{N,m}=S[t_1,\ldots,t_m]/(t_1,\ldots,t_m)^N}$ with~${I_{N,m}=(t_1,\ldots,t_m)/(t_1,\ldots,t_m)^N}$. We show in Subsection~\ref{subsec:reldR} the vanishing of relative de Rham cohomology for such objects (see Proposition~\ref{prop:trivdR}). The vanishing does not hold for an arbitrary split nilpotent extension as we discuss at the end of Subsection~\ref{subsec:reldR}, where we also note a connection of this problem to the Milnor and Tyurina numbers. In Subsection~\ref{subsec:ffa}, we define finitely freely approximable functors from $\SNilp_N(S)$ to the category of abelian groups (see Definition~\ref{def:frapp}). This notion is useful for us because a morphism between such functors is an isomorphisms if and only if it induces isomorphisms between their values on finite free split nilpotent extensions (see Lemma~\ref{lem:isomapprox}). In addition, we give sufficient conditions for a functor to be finitely freely approximable (see Proposition~\ref{prop:elappr1}). This is applied in Subsection~\ref{subsec:constrBloch} to show that the functors defined by~${K_n^M(R,I)}$,~${\Omega^n_{R,I}}$, and~${\Omega^n_{R,I}/d\,\Omega^{n-1}_{R,I}}$ are finitely freely approximable (see Lemma~\ref{prop:elaprKOmgega} and Corollary~\ref{cor:nonapprox}). In turn, this is used to construct the Bloch map. We also outline an explicit proof of the existence of the Bloch map at the end of Subsection~\ref{subsec:constrBloch}.

In Section~\ref{sect:powerseries}, we consider the second Milnor $K$-group $K_2^M\big(S[[t]]\big)$ of the ring of formal power series $S[[t]]$ in a formal variable $t$. In Subsection~\ref{subsec:filtrMilnor}, we define a decreasing filtration~$V_p$,~${p\geqslant 0}$, on this group, which is induced by a standard decreasing filtration on the group~$S[[t]]^*$ by subgroups $1+t^pS[[t]]$. The main result of this section claims that if $2p$ is invertible in $S$ and~$S$ is weakly $5$-fold stable, then $V_p$ coincides with its subgroup $K_2^M\big(S[[t]],(t^p)\big)$ (see Proposition~\ref{prop:vanish} and Corollary~\ref{cor:vanish}). The proof of this fact uses an auxiliary ring $S'=S[[x]]$ together with a collection of homomorphisms of algebras over $S[[t]]$ from $S'[[t]]$ to $S[[t]]$ that send $x$ to $t^q$,~${q\geqslant 1}$. They allows us to make induction with respect to indices in the filtration~$V_p$ (see Lemma~\ref{lem:lift}). We also use certain maps from~$\Omega^1_S$ to $K_2^M\big(S[[t]]\big)/V_{p+1}$ constructed in Subsection~\ref{subsec:constrsymbforms} with the help of the main result of~\cite{GorchinskiyOsipov2015Miln}. Subsection~\ref{subsec:key} contains a technical result, which is a key for the whole proof (see Proposition~\ref{prop:key}). Using this result, we prove in Subsection~\ref{subsec:proofvanish} the needed equality ${V_p=K_2^M\big(S[[t]],(t^p)\big)}$.

Section~\ref{sec:proofmain} contains the proof of the main result that the Bloch map is an isomorphism. The proof consists in a series of reductions based on an auxiliary result (see Lemma~\ref{lem:embed}) given in Subsection~\ref{subsec:redlemma}. Using the above equality ${V_p=K_2^M\big(S[[t]],(t^p)\big)}$, we establish in Subsection~\ref{subsec:partialmain} that the Bloch map is an isomorphism for $(\bar t^{\,N-1})\subset R_{N}$, where $R_{N}=S[t]/(t^N)$ and $\bar t$ denotes the image of $t$ in $R_N$ (note that this is a non-split case). Then we deduce from this a special case of the main theorem for $(\bar t\,)\subset R_N$ (see Proposition~\ref{prop:m=1}). We prove the main result in full generality in Subsection~\ref{subsec:proofmain} reducing it to the case $I_{N,m}\subset R_{N,m}$ with the help of finitely freely approximable functors and then further to the case $(\bar t\,)\subset R_N$. Finally, in Subsection~\ref{subsec:nonexistBloch}, we provide a series of examples of a nilpotent ideal $I$ in a ring $R$ such that the quotient $R/I$ does not split out of $R$ and the Bloch map does not exist (see Proposition~\ref{prop:nonBloch}).

\medskip

The authors are grateful to Denis Osipov and Vadim Vologodsky for useful discussions on the subject of the paper. The authors are partially supported by Laboratory of Mirror Symmetry NRU HSE, RF Government grant, ag.~\textnumero~14.641.31.0001

\section{Preliminaries and statements of results}\label{sec:prelim}

Throughout the paper, by a ring we mean a commutative associative unital ring. Let~$R$ be such a ring. If we need auxiliary assumptions on $R$, we say this explicitly in what follows. Let $n\geqslant 0$ be a natural number.

\subsection{Milnor $K$-groups}\label{subsec:Milnorbasic}

By $R^*$ denote the multiplicative group of invertible elements in $R$. The {\it $n$-th Milnor $K$-group} of $R$ is defined as the quotient
$$
K^M_n(R):=(R^*)^{\otimes n}/{\rm St}_n(R)\,.
$$
Here, ${\rm St}_{n}(R)$ is the subgroup of $(R^*)^{\otimes n}$ generated by so-called {\it Steinberg relations}, which are elements of type
$$
r_{1}\otimes\ldots\otimes r_{i}\otimes r\otimes(1-r)\otimes r_{i+1}\otimes\ldots\otimes r_{n-2}\,,
$$
where $0\leqslant i\leqslant n-2$ and $r_1,\ldots,r_{n-2},r,1-r\in R^*$.

For example, $K_0^M(R)=\Z$ and $K^{M}_{1}(R)=R^{*}$. For $n\geqslant 2$, the class in~${K_n^M(R)}$ of a tensor~${r_{1}\otimes\ldots\otimes r_{n}}$ is denoted by~$\{r_{1},\dots,r_{n}\}$. Elements of Milnor \mbox{$K$-groups} are often called {\it symbols}. The group law in $K_n^M(R)$ is written additively except for the case~${n=1}$.

%Product of tensors induces a graded ring structure on $\bigoplus\limits_{i\geqslant 0}K_i^M(R)$.

The assignment of $K_n^M(R)$ to $R$ is functorial with respect to the ring $R$. Given a homomorphism of rings, for simplicity, we denote the corresponding map between their Milnor $K$-groups similarly as the ring homomorphism.

\medskip

Let $I\subset R$ be an ideal such that all elements in $1+I$ are invertible in $R$. Equivalently, an element in $R$ is invertible if and only if its image in~$R/I$ is invertible. An example is when all elements in $I$ are nilpotent or topologically nilpotent. For instance, this holds for $R$ being the ring $S[[t]]$ of formal power series in a formal variable $t$ with coefficients in a ring $S$ and $I=(t)$.

The natural homomorphism $R^*\to (R/I)^*$ is surjective and it follows that the homomorphism $K_n^M(R)\to K_n^M(R/I)$ is surjective as well.

\begin{definition}\label{def:relmil}
The {\it relative $n$-th Milnor \mbox{$K$-group}} is given by the formula
$$
K^{M}_{n}(R,I):={\rm Ker}\big(K_n^M(R)\to K_n^M(R/I)\big)\,.
$$
\end{definition}

In particular, there is an equality $K_1^M(R,I)=1+I$ between subgroups of~${K_1^M(R)=R^*}$.

The following simple lemma is needed for the sequel.

\begin{lemma}\label{lem:generMilnorK}
The relative Milnor $K$-group $K^{M}_{n}(R,I)$ is generated by elements of type $\{r_{1},\ldots,r_{i},1+x,r_{i+1},\ldots,r_{n-1}\}$, where $0\leqslant i\leqslant n-1$, $r_1,\ldots,r_{n-1}\in R^*$, and~$x\in I$.
\end{lemma}
\begin{proof}
By definition of Milnor $K$-groups, we have the following commutative diagram with exact raws and with vertical maps $\alpha$, $\beta$, and $\gamma$ being induced by the quotient map~${R\to R/I}$:
$$
\begin{CD}
0@>>>{\rm St}_{n}(R)@>>> (R^{*})^{\otimes n}@>>> K^{M}_{n}(R)@>>> 0 \\
@. @VV{\alpha}V @VV{\beta}V @VV{\gamma}V \\
0@>>>{\rm St}_n(R/I)@>>>\big((R/I)^{*}\big)^{\otimes n}@>>> K^{M}_{n}(R/I)@>>> 0
\end{CD}
$$
Take an element
$$
a_{1}\otimes\ldots\otimes a_{i}\otimes a\otimes(1-a)\otimes a_{i+1}\otimes\ldots\otimes a_{n-2}\in {\rm St}_n(R/I)\,,
$$
where $0\leqslant i\leqslant n-2$ and $a_1,\ldots,a_{n-2},a,1-a\in (R/I)^*$. Let $r_1,\ldots, r_{n-2},r\in R$ be any preimages of $a_1,\ldots,a_{n-2},a\in R/I$, respectively. Then the elements $r_1,\ldots, r_{n-2},r,1-r$ are invertible in $R$ and we have the equality
$$
\alpha(r_{1}\otimes\ldots\otimes r_{i}\otimes r\otimes(1-r)\otimes r_{i+1}\otimes\ldots\otimes r_{n-2})=
$$
$$
=a_{1}\otimes\ldots\otimes a_{i}\otimes a\otimes(1-a)\otimes a_{i+1}\otimes\ldots\otimes a_{n-2}\,.
$$
Thus the map $\alpha$ is surjective. Therefore, by the snake lemma, the natural map
$$
{\rm Ker}(\beta)\longrightarrow {\rm Ker}(\gamma)=K_n^M(R,I)
$$ is surjective as well.

On the other hand, we have an exact sequence
$$
1\longrightarrow (1+I)\longrightarrow R^*\longrightarrow (R/I)^*\longrightarrow 1\,.
$$
Since tensor product is right exact, we obtain a right exact sequence
$$
\mbox{$\bigoplus\limits_{i=0}^{n-1}$}\,(R^*)^{\otimes i}\otimes (1+I)\otimes (R^*)^{(n-i-1)}\longrightarrow(R^{*})^{\otimes n}\stackrel{\beta}\longrightarrow \big((R/I)^{*}\big)^{\otimes n}\longrightarrow 1\,.
$$
This gives an explicit description of $\Ker(\beta)$, which finishes the proof.
\end{proof}

It follows directly from Definition~\ref{def:relmil} that for any ideal~$J\subset R$ contained in~$I$, there is an exact sequence
\begin{equation}\label{eq:exactMilnor}
0\longrightarrow K_n^M(R,J)\longrightarrow K_n^M(R,I)\longrightarrow K_n^M(R',I')\longrightarrow 0\,,
\end{equation}
where we put $R'=R/J$, $I'=I/J$.

\subsection{Differential forms}\label{subsec:basicdiffforms}

By $\Omega^{1}_{R}$ denote the $R$-module of (absolute) differential forms of $R$. Recall that the~\mbox{$R$-module} $\Omega_R^1$ is generated by elements~$dr$, $r\in R$, subject to linearity $d(r+s)=dr+ds$ and the Leibniz rule ${d(rs)=rds+sdr}$, where $r,s\in R$.

The $R$-module $\Omega^n_R$ of {\it differential forms of degree~$n$} is defined as the wedge power
$$
\Omega^n_R:=\mbox{$\bigwedge^n_R\Omega^1_R$}\,.
$$
By definition, $\Omega^0_R=R$. Explicitly, $\Omega^n_R$ is the quotient of the $R$-module $(\Omega^1_R)_R^{\otimes n}$ over the \mbox{$R$-submodule} generated by elements of type
$$
dr_{1}\otimes\ldots\otimes dr_{i}\otimes dr\otimes dr\otimes dr_{i+1}\otimes\ldots\otimes dr_{n-2}\,,
$$
where $0\leqslant i\leqslant n-2$ and $r_1,\ldots,r_{n-2},r\in R$.

%Wedge product induces a graded ring structure on $\bigoplus\limits_{i\geqslant 0}\Omega^i_R$.

We have a group homomorphism
$$
d\;:\;R\longrightarrow \Omega^{1}_{R}\,,\qquad r\longmapsto dr\,,
$$
which defines also a group homomorphism
$$
d\;:\;\Omega^{n}_{R}\longrightarrow \Omega^{n+1}_{R}\,,\qquad sdr_{1}\wedge\ldots\wedge dr_{n}\longmapsto ds\wedge dr_{1}\wedge\ldots\wedge dr_{n}\,,
$$
called {\it de Rham differential}. Since $d^2=0$, we have a complex
$$
R\stackrel{d}\longrightarrow\Omega^{1}_{R}\stackrel{d}\longrightarrow\ldots\stackrel{d}\longrightarrow\Omega^{i}_{R}\stackrel{d}\longrightarrow\ldots\,,
$$
called {\it de Rham complex} of $R$. Its cohomology groups are called {\it de Rham cohomology} of~$R$ and are denoted by $H_{dR}^i(R)$, $i\geqslant 0$.

By $(\Omega_R^n)^{cl}$ denote the group of closed differential forms, that is, put
$$
(\Omega_R^n)^{cl}:={\rm Ker}\big(d\colon \Omega^n_R\to\Omega^{n+1}_R\big)\,.
$$

The assignment of $\Omega^n_R$ to $R$ is functorial with respect to the ring $R$. As in the case of Milnor $K$-groups, given a homomorphism of rings, we denote the corresponding maps between their modules of differential forms and de Rham cohomology similarly as the ring homomorphism.

\medskip

Let $I\subset R$ be an ideal. The natural morphism of $R$-modules ${\Omega^{n}_{R}\to\Omega^{n}_{R/I}}$ is surjective.

\begin{definition}\label{def:relforms}
The {\it relative $R$-module of differential forms of degree $n$} is given by the formula
$$
\Omega^{n}_{R,I}:={\rm Ker}\big(\Omega^{n}_{R}\to\Omega^{n}_{R/I}\big)\,.
$$
\end{definition}

Note that we use this term by analogy with relative Milnor $K$-groups (see Definition~\ref{def:relmil}). The reader is warned that usually the term ``relative differential forms'' has another meaning in the context of a homomorphism between rings.

In particular, there is an equality $\Omega^0_{R,I}=I$ between $R$-submodules of~${\Omega^0_R=R}$.

Relative modules of differential forms give a subcomplex $\Omega^{\bullet}_{R,I}$ of $\Omega^{\bullet}_R$, called a {\it relative de Rham complex}. Its cohomology groups are denoted by $H^i_{dR}(R,I)$, $i\geqslant 0$, and are called {\it relative de Rham cohomology}.

The following simple lemma is needed for the sequel.

\begin{lemma}\label{lem:generKahler}
The relative module of differential forms $\Omega^{n}_{R,I}$ is generated additively as an abelian group by differential forms of type ${x\,dr_{1}\wedge\ldots\wedge dr_{n}}$ and by differential forms of type ${r_1\,dx\wedge dr_2\wedge\ldots\wedge dr_{n}}$, where ${r_1,\ldots,r_{n}\in R}$ and~${x\in I}$.
\end{lemma}
\begin{proof}
By~\cite[Theor.\,25.2]{Matsumura}, there is an exact sequence
$$
I\longrightarrow \Omega^1_{R}/(I\cdot \Omega^1_{R})\longrightarrow \Omega^1_{R/I}\longrightarrow 0\,,
$$
where the first map sends an element $x\in I$ to the class of $dx$ in the quotient. In other words, there is an isomorphism of $R$-modules
$$
\Omega^{1}_{R/I}\simeq \Omega^{1}_{R}/(I\cdot\Omega^{1}_{R}+dI)\,.
$$
Taking the wedge power, we obtain isomorphisms
$$
\Omega^{n}_{R/I}\simeq \Omega^{n}_{R}/\big((I\cdot\Omega^{1}_{R}+dI)\wedge\Omega^{n-1}_R\big)\simeq \Omega^{n}_{R}/(I\cdot\Omega^{n}_{R}+dI\wedge\Omega^{n-1}_R)\,,
$$
which proves the lemma.
\end{proof}

Similarly to Milnor $K$-groups, given an ideal~$J\subset R$ contained in~$I$, there is an exact sequence of relative de Rham complexes
$$
0\longrightarrow \Omega^{\bullet}_{R,J}\longrightarrow \Omega^{\bullet}_{R,I}\longrightarrow \Omega^{\bullet}_{R',I'}\longrightarrow 0\,,
$$
where, as above, we put $R'=R/J$, $I'=I/J$. This gives a long exact sequence of relative de Rham cohomology
\begin{equation}\label{eq:longseqcohom}
\ldots \longrightarrow H^{n-1}_{dR}(R',I')\longrightarrow H^{n}_{dR}(R,J)\longrightarrow H^{n}_{dR}(R,I)\longrightarrow H^n_{dR}(R',I')\longrightarrow\ldots\,,
\end{equation}
Besides, truncating the relative de Rham complexes in degrees greater than $n$ and taking the corresponding long exact sequence of cohomology, we obtain an exact sequence
\begin{equation}\label{eq:longseqforms}
\ldots \longrightarrow H^{n-1}_{dR}(R',I')\longrightarrow \Omega^{n}_{R,J}/d\,\Omega^{n-1}_{R,J}\longrightarrow \Omega^{n}_{R,I}/d\,\Omega^{n-1}_{R,I}\longrightarrow \Omega^{n}_{R',I'}/d\,\Omega^{n-1}_{R',I'}\longrightarrow 0\,.
\end{equation}

\subsection{Bloch map}\label{subsec:Blochmap}

Recall that $n\geqslant 0$ is a natural number. It is easy to check that there is a homomorphism of groups
$$
d\log\;:\;K^{M}_{n+1}(R)\longrightarrow\Omega^{n+1}_R\,,\qquad \{r_{1},\ldots,r_{n+1}\}\longmapsto \frac{dr_{1}}{r_{1}}\wedge\ldots\wedge\frac{dr_{n+1}}{r_{n+1}}\,,
$$
which is functorial with respect to $R$. The image of $d\log$ is contained in the subgroup~$(\Omega^{n+1}_R)^{cl}\subset \Omega^{n+1}_R$.

Let $I\subset R$ be a nilpotent ideal. Since $d\log$ is functorial, we have a homomorphism between relative groups
$$
d\log\;:\;K^{M}_{n+1}(R,I)\longrightarrow(\Omega^{n+1}_{R,I})^{cl}\,.
$$
Let $N\geqslant 1$ be a natural number such that $I^N=0$. Until the end of this subsection, we assume that $(N-1)!$ is invertible in $R$.

For an element $x\in I$, put
$$
\log(1+x):=x-\frac{x^2}{2}+\ldots+(-1)^{N}\frac{x^{N-1}}{N-1}\in I\,.
$$
Note that there are equalities
$$
d\big(\log(1+x)\big)=\frac{d(1+x)}{1+x}=(d\log)(1+x)\,.
$$

\begin{lemma}\label{lem:exact}
The image of the map~${d\log\colon K^{M}_{n+1}(R,I)\to(\Omega^{n+1}_{R,I})^{cl}}$ is contained in the relative subgroup of exact differential forms~${d\,\Omega^{n}_{R,I}\subset (\Omega^{n+1}_{R,I})^{cl}}$, that is, we have a map
$$
d\log\;:\;K^{M}_{n+1}(R,I)\longrightarrow d\,\Omega^{n}_{R,I}\,.
$$
\end{lemma}
\begin{proof}
For all $i$, $0\leqslant i\leqslant n$, $r_1,\ldots,r_{n}\in R^*$, and $x\in I$, there is an equality
\begin{equation}\label{eq:Blochmap}
d\log\{r_1,\ldots,r_{i},1+x,r_{i+1},\ldots,r_{n}\}=
d\left((-1)^{i}\log(1+x)\frac{dr_{1}}{r_{1}}\wedge\ldots\wedge\frac{dr_{n}}{r_{n}}\right)\,.
\end{equation}
We finish the proof applying Lemma~\ref{lem:generMilnorK} with $n$ replaced by $n+1$.
\end{proof}

Our main object of study is the following map, which was introduced originally by Bloch~\cite[\S\,1]{Bloch75} (previous versions of this map were constructed by van der Kallen~\cite{vdKallen71} and Bloch~\cite{Bloch73}).

\begin{definition}\label{def:Blochmap}
A homomorphism of groups
$$
{\rm B}\;:\;K^{M}_{n+1}(R,I)\longrightarrow\Omega^{n}_{R,I}/d\,\Omega^{n-1}_{R,I}
$$
is called a {\it Bloch map} if for all $i$, $0\leqslant i\leqslant n$, $r_{1},\ldots,r_{n}\in R^{*}$, and $x\in I$, we have (cf. formula~\eqref{eq:Blochmap})
\begin{equation}\label{eq:defBloch}
{\rm B}\,\{r_1,\ldots,r_{i},1+x,r_{i+1},\ldots,r_{n}\}=
(-1)^{i}\log(1+x)\frac{dr_{1}}{r_{1}}\wedge\ldots\wedge\frac{dr_{n}}{r_{n}}\,,
\end{equation}
where, for simplicity, we denote similarly elements in $\Omega^n_{R,I}$ and their images under the quotient map $\Omega^n_{R,I}\to \Omega^n_{R,I}/d\,\Omega^{n-1}_{R,I}$.
\end{definition}

Sometimes we denote the Bloch map as in Definition~\ref{def:Blochmap} by $\B_{R,I}$ to specify the ring and the ideal. One can consider the Bloch map as an integral of the map $d\log$ from Lemma~\ref{lem:exact}.

\medskip

Let us discuss some general properties of the Bloch map. By Lemma~\ref{lem:generMilnorK}, a Bloch map is unique whenever it exists.  For $n=0$, the Bloch map always exists and coincides with the isomorphism~${\log\colon 1+I\stackrel{\sim}\longrightarrow I}$.

\begin{lemma}\label{lem:surjBlochmap}
Suppose that any element in $R$ is a sum of invertible elements and that the Bloch map $\B\colon K^{M}_{n+1}(R,I)\to\Omega^{n}_{R,I}/d\,\Omega^{n-1}_{R,I}$ exists. Then the Bloch map is surjective.
\end{lemma}
\begin{proof}
For all ${r_1,\ldots,r_{n}\in R}$ and~${x\in I}$, there is an equality in $\Omega^n_{R,I}$
$$
{r_1\,dx\wedge dr_2\wedge\ldots\wedge dr_{n}}=d(r_1x\,dr_2\wedge\ldots\wedge dr_n)-x\,dr_1\wedge dr_2\wedge\ldots\wedge dr_n\,.
$$
Hence by Lemma~\ref{lem:generKahler}, it is enough to show that classes in $\Omega^n_{R,I}/d\,\Omega^{n-1}_{R,I}$ of differential forms of type $x\,dr_1\wedge \ldots\wedge dr_n$ are in the image of the Bloch map.

Since $R$ is generated additively by invertible elements, it is enough to consider the case when $r_1,\ldots,r_n\in R^*$. Then there is an equality
$$
\B\,\{\exp(xr_1\ldots r_n),r_1,\ldots,r_n\}=x\,dr_1\wedge\ldots\wedge dr_n\,,
$$
where for an element $y\in I$, we put
$$
\exp(y):=1+y+\frac{y}{2}+\ldots+\frac{y^{N-1}}{(N-1)!}\in R\,.
$$
This proves the lemma.
\end{proof}

The following simple observation claims the existence of the Bloch map in some special cases.

\begin{remark}\label{rmk:Blochmap}
The Bloch map~${{\rm B}\colon K^{M}_{n+1}(R,I)\to\Omega^{n}_{R,I}/d\,\Omega^{n-1}_{R,I}}$ exists when there is a vanishing~${H^n_{dR}(R,I)=0}$. The reason is that the latter condition is equivalent to requiring that the de Rham differential gives an isomorphism
$$
d\;:\;\Omega^n_{R,I}/d\,\Omega^{n-1}_{R,I}\stackrel{\sim}\longrightarrow d\,\Omega^{n}_{R,I}\,.
$$
The Bloch map is equal to the composition of the map $d\log$ from Lemma~\ref{lem:exact} and the inverse of this isomorphism.
\end{remark}

Consider an example, which was treated also in~\cite{GorchinskiyOsipov2015Miln}.

\begin{example}\label{examp:Blochmap2}
Let $S$ be a ring such that~$2$ is invertible in it. Let $\varepsilon$ be a formal variable such that $\varepsilon^2=0$.  Then the ring $R=S[\varepsilon]$ and the ideal ${I=(\varepsilon)}$ satisfy $H^n_{dR}(R,I)=0$ for any $n\geqslant 0$ (see Lemma~\ref{lem:trivdR} and Proposition~\ref{prop:trivdR} below for generalizations of this fact). Indeed, using the equalities ${\varepsilon\, d\varepsilon =\frac{1}{2}d(\varepsilon^2)=0}$, one shows that there is a decomposition
$$
\Omega^{n+1}_{S[\varepsilon],(\varepsilon)}\simeq (\varepsilon\, \Omega^{n+1}_S)\oplus (d\varepsilon \wedge \Omega^n_S)\,.
$$
This implies that there are isomorphisms
\begin{equation}\label{eq:isom}
\Omega^n_S\stackrel{\sim}\longrightarrow \Omega^n_{S[\varepsilon],(\varepsilon)}/d\,\Omega^{n-1}_{S[\varepsilon],(\varepsilon)}\,,\qquad \omega\longmapsto \varepsilon\,\omega\,,
\end{equation}
$$
\Omega^n_S\stackrel{\sim}\longrightarrow d\,\Omega^{n}_{S[\varepsilon],(\varepsilon)}\,,\qquad \omega\longmapsto d(\varepsilon\,\omega)=\varepsilon\, d\omega+d\varepsilon\wedge\omega\,.
$$
Therefore, there is an isomorphism
$$
d\;:\;\Omega^n_{S[\varepsilon],(\varepsilon)}/d\,\Omega^{n-1}_{S[\varepsilon],(\varepsilon)}\stackrel{\sim}\longrightarrow d\,\Omega^{n}_{S[\varepsilon],(\varepsilon)}\,.
$$
Thus by Remark~\ref{rmk:Blochmap}, the Bloch map exists in this case.

Moreover, one checks directly that the Bloch map coincides with the composition of the map ${K_{n+1}^M\big(S[\varepsilon],(\varepsilon)\big)\to \Omega^n_{S}}$ constructed in~\cite[Def.\,2.7]{GorchinskiyOsipov2015Miln} and the isomorphism~\eqref{eq:isom}. In particular, for $a\in S$ and ${b_1,\ldots,b_n\in S^*}$, we have (cf.~\cite[Exam.\,2.8]{GorchinskiyOsipov2015Miln})
$$
{\rm B}\,\{1+ab_1\ldots b_n\varepsilon,b_1,\ldots,b_n\}=\varepsilon\, adb_1\wedge\ldots\wedge db_n\in \Omega^n_{S[\varepsilon],(\varepsilon)}/d\,\Omega^{n-1}_{S[\varepsilon],(\varepsilon)}\,.
$$
\end{example}

\subsection{Main results}\label{subsec:mainresults}

The following statement asserts the existence of the Bloch map in the case when the quotient~$R/I$ splits out of $R$.

\begin{theorem}\label{the:Bloch}
Let $I\subset R$ be a nilpotent ideal and $N\geqslant 1$ be a natural number such that $I^N=0$. Suppose that the quotient map~${R\to R/I}$ admits a splitting by a ring homomorphism~${R/I\to R}$ and that~$N!$ is invertible in~$R$. Then for any natural number $n\geqslant 0$, the Bloch map
$$
{\rm B}\;:\;K^{M}_{n+1}(R,I)\longrightarrow\Omega^{n}_{R,I}/d\,\Omega^{n-1}_{R,I}
$$
exists.
\end{theorem}

Theorem~\ref{the:Bloch} is proved in Section~\ref{sec:Bloch}.

\medskip

In order to state our main result, we need to introduce the following notion, which goes back to Morrow~\cite[Def.\,3.1]{Morrow}.

\begin{definition}\label{def:weakst}
Given a natural number $k\geqslant2$, a ring $R$ is called {\it weakly \mbox{$k$-fold} stable} if for any collection of elements  $r_{1},\ldots, r_{k-1}\in R$,
there exists $r\in R^{*}$ such that ${r_{1}+r,\ldots, r_{k-1}+r\in R^*}$.
\end{definition}

For example, a ring $R$ is weakly $2$-fold stable if and only if any element in~$R$ is a sum of two invertible elements. It is easy to see that, given an ideal $I\subset R$ such that all elements in $1+I$ are invertible in $R$, the quotient $R/I$ is weakly $k$-fold stable if and only if the initial ring $R$ is weakly $k$-fold stable. In particular, $S$ is a weakly $k$-fold stable ring if and only if the same holds for~$S[[t]]$. See more details on weak stability, e.g., in~\cite[\S\,2.2]{GorchinskiyOsipov2015Miln}.

The pair $(R,I)$ as in Theorem~\ref{the:Bloch} is a split nilpotent extension of the quotient~$R/I$ (see Definition~\ref{def:splitnilp} below). Our main result claims that the Bloch map is an isomorphism for split nilpotent extensions with sufficiently many invertible elements.

\begin{theorem}\label{thm:main}
Let $I\subset R$ be a nilpotent ideal and $N\geqslant 1$ be a natural number such that $I^N=0$. Suppose that the quotient map~${R\to R/I}$ admits a splitting by a ring homomorphism~${R/I\to R}$, that~$N!$ is invertible in~$R$, and that $R$ is weakly $5$-fold stable. Then for any natural number $n\geqslant 0$, the Bloch map is an isomorphism
$$
{\rm B}\;:\;K_{n+1}^M(R,I)\stackrel{\sim}\longrightarrow \Omega^n_{R,I}/d\,\Omega^{n-1}_{R,I}\,.
$$
\end{theorem}

Theorem~\ref{thm:main} is proved in Section~\ref{sec:proofmain}. Note that Gorchinskiy and \mbox{Osipov}~\cite[Theor.\,2.9]{GorchinskiyOsipov2015Miln} proved a special case of Theorem~\ref{thm:main} when $R=S[\varepsilon]$, $I=(\varepsilon)$, where $\varepsilon^2=0$ and $S$ is a ring such that $2$ is invertible in it and $S$ a weakly $5$-fold stable (see also Example~\ref{examp:Blochmap2}).

\medskip

We show in Proposition~\ref{prop:nonBloch} below that the Bloch map need not exist if we drop the assumption that $R/I$ splits out of~$R$. So, Theorem~\ref{the:Bloch} and henceforth Theorem~\ref{thm:main} fail in a non-split case.

Moreover, it is easy to see that there is no a functorial isomorphism between~${K_{n+1}^M(R,I)}$ and~${\Omega^n_{R,I}/d\,\Omega^{n-1}_{R,I}}$ when $R/I$ does not necessarily split out of~$R$. Indeed, given an ideal ${J\subset R}$ contained in~$I$, one has an exact sequence~\eqref{eq:exactMilnor} from Subsection~\ref{subsec:Milnorbasic} for relative Milnor $K$-groups, while for relative modules of differential forms one has an exact sequence~\eqref{eq:longseqforms} from Subsection~\ref{subsec:basicdiffforms} with, possibly, a non-trivial term~${H^{n-1}_{dR}(R',I')}$, where~${R'=R/J}$ and~${I'=I/J}$. This argument is also used in the proof of Proposition~\ref{prop:nonBloch}.

\begin{remark}\label{rem:nonsplit}
An interesting problem is to find an alternative formulation of Theorem~\ref{thm:main} that would be valid in a non-split case as well. One way might be to replace~${\Omega^n_{R,I}/d\,\Omega^{n-1}_{R,I}}$ with the group
$$
{\rm Im}\big(\Omega^n_{R,I}/d\,\Omega^{n-1}_{R,I}\to \Omega^n_{R}/d\,\Omega^{n-1}_{R}\big)=
\Ker\big(\Omega^n_{R}/d\,\Omega^{n-1}_{R}\to \Omega^n_{R/I}/d\,\Omega^{n-1}_{R/I}\big)\,.
$$
This group coincides with $\Omega^n_{R,I}/d\,\Omega^{n-1}_{R,I}$ when $R/I$ splits out of $R$.

A more sophisticated way, which also goes along with Goodwillie's theorem~\cite{Goodwillie}, is to consider the groups~${K_{n+1}^M(R,I)}$ and~${\Omega^n_{R,I}/d\,\Omega^{n-1}_{R,I}}$ in Theorem~\ref{thm:main} as degree zero cohomology of certain (non-positively graded) complexes. Namely, one can show that the group~$\Omega^n_{R}/d\,\Omega^{n-1}_{R}$ is isomorphic to the degree zero cohomology group of the complex
$$
{\rm cone}\big(F^{n+1}\LL\Omega^{\bullet}_{R}\to \LL\Omega^{\bullet}_{R}\big)[n]\,,
$$
where $\LL\Omega^{\bullet}_R$ is the derived de Rham complex of $R$ and $F^{\bullet}$ is the Hodge filtration on it. Thus a natural substitute for ${\Omega^n_{R,I}/d\,\Omega^{n-1}_{R,I}}$ is the degree zero cohomology group of the complex
$$
{\rm cone}\big(F^{n+1}\LL\Omega^{\bullet}_{R,I}\to \LL\Omega^{\bullet}_{R,I}\big)[n]\,,
$$
where
$$
\LL\Omega^{\bullet}_{R,I}\simeq {\rm cone}\big(\LL\Omega^{\bullet}_R\to \LL\Omega^{\bullet}_{R/I}\big)[-1]
$$
is the relative derived de Rham complex. Again, this group coincides with~$\Omega^n_{R,I}/d\,\Omega^{n-1}_{R,I}$ when $R/I$ splits out of $R$.

It is not clear what should be a right complex for Milnor $K$-groups. It seems possible that this might involve a version for commutative simplicial rings of Goncharov's complexes~\cite{Goncharov}, giving sort of derived Milnor $K$-groups.
\end{remark}

\section{Proof of Theorem~\ref{the:Bloch}}\label{sec:Bloch}

Fix a ring~$S$. Let $N\geqslant 1$ be a natural number, which will be the nilpotency degree.

\subsection{Split nilpotent extensions}\label{subsec:splitnilp}

We will work with the following objects.

\begin{definition}\label{def:splitnilp}
\hspace{0cm}
\begin{itemize}
\item[(i)]
A {\it split nilpotent extension of~$S$ of nilpotency degree $N$} is a pair $(R,I)$, where $R$ is an \mbox{$S$-algebra} and~${I\subset R}$ is a nilpotent ideal such that $I^N=0$ and the summation map $S\oplus I\to R$ is an isomorphism of $S$-modules.
\item[(ii)]
A {\it morphism of split nilpotent extensions} $(R,I)\to (R',I')$ is a morphism of $S$-algebras $f\colon R\to R'$ such that~${f(I)\subset I'}$.
\item[(iii)]
By
$$
\SNilp_N(S)
$$
denote the category of split nilpotent extensions of $S$ of nilpotency degree~$N$.
\end{itemize}
\end{definition}

In particular, a $S$-algebra $R$ as in Definition~\ref{def:splitnilp}(i) is augmented, that is, it is fixed a (surjective) homomorphism of $S$-algebras $R\to S$ whose kernel is~$I$. Note that giving a split nilpotent extension of $S$ of nilpotency degree $N$ is the same as giving a $S$-module $I$ with a commutative associative product map~$I\otimes_S I\to I$ of nilpotency degree $N$. Indeed, such $S$-module $I$ corresponds to the split nilpotent extension $(S\oplus I,I)$ of $S$.

A morphism of split nilpotent extensions is the same as a homomorphism of augmented $S$-algebras.

\medskip

We will use the following notation: if a ring $R$ is a quotient of the ring~$S[t_1,\ldots,t_m]$ of polynomials in formal variables $t_1,\ldots,t_m$, $m\geqslant 1$, we denote the image of~$t_i$ in $R$ by~$\bar t_i$ for each $i$, ${1\leqslant i \leqslant m}$.

%We will particularly use the following split nilpotent extensions of nilpotency degree $N$.

\begin{definition}\label{def:elem}
An object in $\SNilp_N(S)$ is {\it finite free} if it is isomorphic to
$$
(R_{N,m},I_{N,m}):=\big(S[t_{1},\ldots,t_m]/(t_1,\ldots,t_m)^N\,,(\bar t_1,\ldots,\bar t_m)\big)
$$
for some natural number $m\geqslant 0$.
\end{definition}

Split nilpotent extensions as in Definition~\ref{def:elem} are indeed finite free objects in~$\SNilp_N(S)$ in the following sense: they are values on finite sets of the left adjoint functor to the functor from~$\SNilp_N(S)$ to the category of sets that sends~$(R,I)$ to the ideal $I$ considered as a set.

\subsection{Vanishing of relative de Rham cohomology}\label{subsec:reldR}

%In what follows we consider a standard graded structure on polynomial algebras $S[t_1,\ldots,t_m]$ such that elements of %$S$ have degree $0$ and $t_i$ has degree $1$ for any $i$, ${1\leqslant i\leqslant m}$. Given a (surjective) homomorphism %$S[t_1,\ldots,t_m]\to R$ to a ring $R$, we denote by~$\bar t_i$ the image of~$t_i$ in $R$.

Proposition~\ref{prop:trivdR} below claims the vanishing of relative de Rham cohomology (see Subsection~\ref{subsec:basicdiffforms}) for finite free split nilpotent extensions under an additional invertibility condition. The proof of this fact is based on the following lemma, whose main argument is the action of the Euler vector field on differential forms by a Lie derivative.

\begin{lemma}\label{lem:trivdR}
Let $J\subset S[t]$ be an ideal such that ${t^N\in J}$ and $J$ is generated by monomials in $t$ (which may be of degree zero, that is, be elements of $S$). Let $R=S[t]/J$ and $I=(\bar t\,)\subset R$. Suppose that $N!$ is invertible in $S$. Then for all~$n\geqslant 0$, we have $H^n_{dR}(R,I)=0$ .
\end{lemma}
\begin{proof}
Consider a standard graded ring structure on $S[t]$ such that elements of~$S$ have degree zero and $t$ has degree one. The de Rham complex~$\Omega^{\bullet}_{S[t]}$ becomes naturally a graded complex. We call the degree that corresponds to this grading an internal degree in order to distinguish it from the degree of differential forms. Explicitly, for all $n,i\geqslant 0$, the homogenous component~$\big(\Omega^n_{S[t]}\big)_{i}$ in
$$
\Omega^n_{S[t]}\simeq \Omega_S^n[t]\oplus \Omega_S^{n-1}[t]dt
$$
of internal degree $i$ consists of differential forms of type
\begin{equation}\label{eq:formform}
\mbox{$\omega\cdot t^i+t^{i-1}\eta\wedge dt$}\,,
\end{equation}
where $\omega\in \Omega^n_S$ and $\eta\in \Omega^{n-1}_S$. In particular, we have $\big(\Omega^n_{S[t]}\big)_{0}=\Omega_S^n$ for $n\geqslant 0$.

For all $n,i\geqslant 0$, define a $S$-linear map
$$
h\;:\;\big(\Omega^{n}_{S[t]}\big)_{i}\longrightarrow\big(\Omega^{n-1}_{S[t]}\big)_{i}\,,\qquad
h(\mbox{$\omega\cdot t^i+t^{i-1}\eta\wedge dt$})=
(-1)^{n-1}\eta\cdot t^{i}\,.
$$
By definition, $h$ vanishes on $\Omega^0_{S[t]}$ and on $\big(\Omega^{n}_{S[t]}\big)_{0}$, where $n\geqslant 1$. One checks directly that the restriction of the map $d\circ h+h\circ d$ to the homogenous component $\big(\Omega^{n}_{S[t]}\big)_i$ coincides with multiplication by $i$.

Note that an ideal in $S[t]$ is generated by monomials if and only if it is homogenous. In particular, the ideal $J$ is homogenous. Hence, we obtain a graded ring structure on $R$, which induces a graded complex structure on the de Rham complex $\Omega^{\bullet}_R$. The quotient map ${\Omega^{\bullet}_{S[t]}\to \Omega^{\bullet}_R}$ preserves the internal degree and $\big(\Omega^n_R\big)_i$ is the quotient of the $S$-module~$\big(\Omega^{n}_{S[t]}\big)_{i}$ for all $n,i\geqslant 0$. We claim that there is an equality
\begin{equation}\label{eq:decompRI}
\Omega^{\bullet}_{R,I}=\mbox{$\bigoplus\limits_{i=1}^{N-1}\big(\Omega^{\bullet}_R\big)_{i}$}\,.
\end{equation}
Indeed, formula~\eqref{eq:formform} implies that $\big(\Omega^{\bullet}_R\big)_i\subset \Omega^{\bullet}_{R,I}$ for any $i\geqslant 1$. In addition, the quotient~$R/I$ splits out of $R$ as the homogenous component of degree zero and we have an isomorphism $\big(\Omega^{\bullet}_R\big)_0\simeq \Omega^{\bullet}_{R/I}$. Therefore, ${\Omega^{\bullet}_{R,I}= \bigoplus\limits_{i\geqslant 1}\big(\Omega^{\bullet}_R\big)_i}$\,. On the other hand, there is an equality
\begin{equation}\label{eq:decompRI2}
0=d(\bar t\,^N)=N\bar t\,^{N-1}d\bar t
\end{equation}
in $\Omega^1_R$. Since $N$ is invertible in $S$ and in $R$, this implies that $\big(\Omega^{\bullet}_{R}\big)_i=0$ for all~$i\geqslant N$.

Furthermore, for each $i\geqslant 0$, the map ${h\colon \big(\Omega^{n}_{S[t]}\big)_{i}\to\big(\Omega^{n-1}_{S[t]}\big)_{i}}$ descends to a map $\bar h\colon{\big(\Omega^n_{R}\big)_i\to \big(\Omega^{n-1}_{R}\big)_i}$. One way to show this is just to check it directly using the description of the kernel of the quotient map ${\big(\Omega^n_{S[t]}\big)_i\to \big(\Omega^n_{R}\big)_i}$ given by Lemma~\ref{lem:generKahler}. A more conceptual way is to observe that $\bar h$ is the contraction with the derivation~$\partial$ on $R$ that acts as multiplication by $i$ on the homogenous component of degree $i$ in $R$ (sometimes $\partial$ is called an Euler vector field).

Clearly, the equality $d\circ h+h\circ d=i$ implies the equality $d\circ \bar h+\bar h\circ d=i$. Alternatively, $d\circ \bar h+\bar h\circ d$ is the Lie derivative defined by the derivation $\partial$, which is known to act as multiplication by $i$ on differential forms of internal degree $i$ in $\Omega^n_R$.

We see that for each $i\geqslant 0$, multiplication by $i$ on the complex $\big(\Omega^{\bullet}_{R}\big)_i$ is homotopically trivial. Using that $(N-1)!$ is invertible in $S$ and formula~\eqref{eq:decompRI}, we conclude that the complex $\Omega^{\bullet}_{R,I}$ is homotopically trivial, whence its cohomology groups vanish.
\end{proof}

\begin{proposition}\label{prop:trivdR}
Suppose that $N!$ is invertible in $S$. Then for all~$m,n\geqslant 0$, we have $H^n_{dR}(R_{N,m},I_{N,m})=0$.
\end{proposition}
\begin{proof}
The proof is by induction on $m$. The case $m=0$ is trivial as the corresponding relative module of differential forms just vanishes.

Let us make an induction step from $m-1$ to $m$. We have natural isomorphisms
$$
R_{N,m-1}\simeq R_{N,m}/(\bar t_m)\,,\qquad I_{N,m-1}\simeq I_{N,m}/(\bar t_m)\,.
$$
Hence the embedding of ideals $(\bar t_m)\subset I_{N,m}$ in the ring $R_{N,m}$ gives an exact sequence of relative de Rham complexes
\begin{equation}\label{eq:exactdR}
0\longrightarrow \Omega^{\bullet}_{R_{N,m},(\bar t_m)}\longrightarrow\Omega^{\bullet}_{R_{N,m},I_{N,m}}\longrightarrow \Omega^{\bullet}_{R_{N,m-1},I_{N,m-1}}\longrightarrow 0\,.
\end{equation}
By the induction hypothesis, we have $H^n_{dR}(R_{N,m-1},I_{N,m-1})=0$ for all $n\geqslant 0$.

Put $S'=S[t_1,\ldots,t_{m-1}]$ and
$$
J=(t_1,\ldots,t_m)^N\subset S[t_1,\ldots,t_m]=S'[t_m]\,.
$$
Clearly, we have~${t_m^N\in J}$. Since the ideal $J$ is generated by monomials in~${t_1,\ldots,t_m}$ with coefficients in~$S$, we see that, in particular, $J$ is generated by monomials in $t_m$ with coefficients in $S'$. Thus by Lemma~\ref{lem:trivdR} applied to $J\subset S'[t_m]$, we see that $H^n_{dR}\big(R_{N,m},(\bar t_m)\big)=0$ for all $n\geqslant 0$.

Using the long exact sequence of cohomology associated with the exact sequence~\eqref{eq:exactdR}, we complete the proof.
\end{proof}

As the second part of the following remark claims, given an arbitrary graded split nilpotent extension of $S$ of nilpotency degree $N$, its relative de Rham cohomology vanish if one requires that more natural numbers than just $N!$ are invertible in $S$.

\begin{remark}
\hspace{0cm}
\begin{itemize}
\item[(i)]
Using the same argument as in the proof of Proposition~\ref{prop:trivdR}, one shows the following generalization of both Lemma~\ref{lem:trivdR} and Proposition~\ref{prop:trivdR}. Let $J\subset S[t_1,\ldots,t_m]$ be an ideal such that ${t_1^N,\ldots,t_m^N\in J}$ and $J$ is generated by monomials in $t_1,\ldots,t_m$. Put
$$
R=S[t_1,\ldots,t_m]/J\,,\qquad I=(\bar t_1,\ldots,\bar t_m)\subset R\,.
$$
Suppose that $N!$ is invertible in $S$. Then $H^n_{dR}(R,I)=0$ for all~$n\geqslant 0$.
\item[(ii)]
Let $R$ be a graded \mbox{$S$-algebra} such that $R_0=S$ and the ideal $I=\bigoplus\limits_{i\geqslant 1}R_i$ satisfies the condition~${I^N=0}$. For simplicity, assume that $R$ is generated as an \mbox{$S$-algebra} by~$m$ homogenous elements of degree one. Then the relative de Rham complex~$\Omega^{\bullet}_{R,I}$ may have non-zero homogenous components of internal degree up to $N+m-2$ (the component of degree $N+m-1$ vanishes by equalities analogous to~\eqref{eq:decompRI2}). For example, the relative de Rham complex~${\Omega^{\bullet}_{R_{N,m},I_{N,m}}}$ has a non-zero component of internal degree~${N+m-2}$. Thus, applying the same argument with the action of the Euler field on differential forms by the Lie derivative as in the proof of Lemma~\ref{lem:trivdR}, one shows the vanishing of relative de Rham cohomology under the condition that $(N+m-2)!$ is invertible in $S$.
\end{itemize}
\end{remark}

Note that relative de Rham cohomology can be non-trivial for an arbitrary split nilpotent extension of $S$, even when $S$ is a $\Q$-algebra. Here is a way to construct such examples.

Let $f\in \Q[t_1,\ldots,t_m]$ be a polynomial such that $f$ defines an isolated singularity at the origin or, equivalently, such that the Milnor ring
$$
R=\Q[[t_1,\ldots,t_m]]/(\partial_{t_1}f,\ldots,\partial_{t_m}f)
$$
is finite-dimensional over $\Q$. The {\it Milnor number} of $f$ (at the origin) is defined as
$$
{\mu(f):=\dim_{\Q}(R)}\,.
$$
Equivalently, we have ${\mu(f)=\dim_{\Q}\big(\Omega^m/df\wedge \Omega^{m-1}\big)}$, where, for short, we put
$$
\Omega^1:=\Q[[t_1,\ldots,t_m]]dt_1\oplus\ldots\oplus \Q[[t_1,\ldots,t_m]]dt_m\,,
$$
$$
\Omega^i:=\mbox{$\bigwedge^i_{\Q[[t_1,\ldots,t_m]]}\Omega^1$}\,,\qquad i\geqslant 0\,.
$$
Recall that the {\it Tyurina number} (see~\cite{Tjurina}) of $f$ is defined as
$$
\tau(f):=\dim_{\Q}\big(R/(\bar f)\big)\,,
$$
where $\bar f$ is the image in $R$ of $f$. Equivalently, we have ${\tau(f)=\dim_{\Q}\big(\Omega^m/(f\cdot\Omega^m+df\wedge \Omega^{m-1})\big)}$. Clearly,~${\mu(f)\geqslant \tau(f)}$.

Suppose that there is a strict inequality
\begin{equation}\label{eq:miltyu}
\mu(f)>\tau(f)\,,
\end{equation}
or, equivalently, that $\bar f\ne 0$. Then $\bar f$ gives a non-zero class in $H^0_{dR}(R,I)$, where $I=(\bar t_1,\ldots,\bar t_m)\subset R$. For example, Grauert and Kerner~\cite[\S\,1.3]{GrauertKerner} showed explicitly that this holds for $f=t_1^4+t_1^2t_2^3+t_2^5$. A criterion for $f$ to satisfy inequality~\eqref{eq:miltyu} was obtained by Saito~\cite{Saito}.

Actually, inequality~\eqref{eq:miltyu} implies that another split nilpotent extension of~$\Q$ has non-trivial relative de Rham cohomology in higher degree. Namely, put
$$
R'=\Q[[t_1,\ldots,t_m]]/\big(f,(t_1,\ldots,t_m)^N\big)
$$
for a sufficiently large natural number~$N$. There is an equality (cf. Lemma~\ref{lem:generKahler})
$$
\tau(f)=\dim_{\Q}(\Omega^m_{R'})\,.
$$
It was proved independently by Palamodov~\cite{Palamodov} and Milnor~\cite[Theor.\,7.2]{Milnor68} that $\mu(f)$ equals the dimension of the space of vanishing cycles associated with~$f$. Malgrange~\cite[Th\'eor.\,5.1, Th\'eor.\,3.7]{Malgrange} (see also a nice survey by Mond~\cite{Mond}) gave another proof of this fact and also proved the equality
$$
\mu(f)=\dim_{\Q} \big(\Omega_{R'}^{m-1}/d\,\Omega_{R'}^{m-2}\big)\,.
$$
This was generalized later to complete intersections with isolated singularities by Tr\'ang~\cite{Trang} and Greuel~\cite[Prop.\,5.1]{Greuel}.

Since the de Rham differential ${d\colon \Omega^{m-1}\to \Omega^{m}}$ is surjective, the de Rham differential ${d\colon \Omega^{m-1}_{R'}\to \Omega^{m}_{R'}}$ is surjective as well. Hence, we have an exact sequence
$$
0\longrightarrow H^{m-1}_{dR}(R')\longrightarrow \Omega_{R'}^{m-1}/d\,\Omega_{R'}^{m-2}\stackrel{d}\longrightarrow \Omega^m_{R'}\longrightarrow 0\,.
$$
Thus, the vanishing $H^i_{dR}(\Q)=0$, $i>0$, yields the equality
$$
\mu(f)-\tau(f)=\dim_{\Q}H^{m-1}_{dR}(R',I')\,,
$$
where $I'=(\bar t_1,\ldots,\bar t_m)\subset R'$. Therefore inequality~\eqref{eq:miltyu} implies that there is a non-zero class in $H^{m-1}_{dR}(R',I')$. In particular, such examples were found by Reiffen~\cite[Satz\,4, Satz\,5]{Reiffen67} explicitly for $m=2,3$ and by Arapura and Kang~\cite[Exam.\,4.4]{ArapuraKang} for $m=2$ with the help of the Maple subroutines of Rossi and Teraccini~\cite{RossiTerracini} to calculate Milnor and Tyurina numbers.

\subsection{Finitely freely approximable functors}\label{subsec:ffa}

%Free split nilpotent extensions of $S$ form a small full subcategory of $\SNilp_N(S)$, which we denote by $\Elem_N(S)$.

%There are many other split nilpotent extensions of $S$ that also look quite simple, for example, the $S$-algebra %$S[t_1,\ldots,t_m]/(t_1,\ldots,t_m)^N$ with the ideal $(t_1,\ldots,t_m)$. However, we prefer to work with elementary %extensions from Definition~\ref{def:elem} because they allow to make induction on $m$ easily. An instance of this is %Proposition~\ref{prop:trivdR} below.

%\medskip

%Let $(R,I)$ be an object in $\SNilp_N(S)$. By $\Elem_N(S)/(R,I)$ denote the category whose objects are $(R',I')$ in %$\Elem_N(S)$ together with a morphism $(R',I')\to (R,I)$ in $\SNilp_N(S)$. Morphisms in~$\Elem_N(S)/(R,I)$ are defined %naturally as morphisms in $\Elem_N(S)$ that commute with the chosen morphisms to $(R,I)$.

By $\Ab$ denote the category of abelian groups.

%Given a functor $F\colon \SNilp_N(S)\to\Ab$, taking the composition of the natural functor $\Elem_N(S)/(R,I)\to %\SNilp_N(S)$ with $F$, we obtain a diagram of abelian groups. Denote by $\underset{\Elem_N(S)/(R,I)}\colim\, F$ the %colimit of this diagram.

\begin{definition}\label{def:frapp}
A functor $F\colon \SNilp_N(S)\to\Ab$ is {\it finitely freely approximable} if for any object~$(R,I)$ in $\SNilp_N(S)$, the natural homomorphism of groups
$$
\underset{(R',I')\to(R,I)}\colim\, F(R',I')\longrightarrow F(R,I)
$$
is an isomorphism, where the colimit is taken with respect to the category of finite free objects in $\SNilp_N(S)$ over $(R,I)$.
\end{definition}

Recall that the category of finite free objects in $\SNilp_N(S)$ over $(R,I)$ is defined as follows. Objects are arrows $f\colon (R',I')\to (R,I)$ in $\SNilp_N(S)$, where $(R',I')$ is a finite free object in $\SNilp_N(S)$. Morphisms from ${f_1\colon (R'_1,I'_1)\to (R,I)}$ to ${f_2\colon (R'_2,I'_2)\to (R,I)}$ are arrows $\varphi\colon (R_1',I_1')\to (R'_2,I'_2)$ in $\SNilp_N(S)$ such that $f_1=f_2\circ\varphi$.

Note that, in general, the colimit in Definition~\ref{def:frapp} is not filtered. Finitely freely approximable functors are useful because of the following obvious observation.

\begin{lemma}\label{lem:isomapprox}
Let $\rho\colon F\to G$ be a morphism of functors from $\SNilp_N(S)$ to~$\Ab$. Suppose that~$F$ and~$G$ are finitely freely approximable and that for any~${m\geqslant 0}$, the corresponding homomorphism ${F(R_{N,m},I_{N,m})\to G(R_{N,m},I_{N,m})}$ is an isomorphism. Then the morphism~$\rho$ is an isomorphism as well.
\end{lemma}
\begin{proof}
Indeed, an isomorphism between diagrams induces an isomorphism between their colimits.
\end{proof}

Given a functor $F\colon\SNilp_N(S)\to\Ab$, define the functor~$F^{fa}$ by the formula
$$
F^{fa}\;:\;\SNilp_N(S)\to\Ab\,,\qquad (R,I)\longmapsto\underset{(R',I')\to(R,I)}\colim\, F(R',I')\,,
$$
where the colimit is as in Definition~\ref{def:frapp}. One checks easily that the functor~$F^{fa}$ is finitely freely approximable and that the values of~$F^{fa}$ and $F$ coincide on finite free objects in $\SNilp_N(S)$. The natural morphism $F^{fa}\to F$ is an isomorphism if and only if $F$ is finitely freely approximable.

Given a morphism of functors $\rho\colon F\to G$, we have a morphism between finitely freely approximable functors $\rho^{fa}\colon F^{fa}\to G^{fa}$. Explicitly, $\rho^{fa}$ is the colimit of $\rho$ over finite free objects in $\SNilp_N(S)$. The assignment of $F^{fa}$ to $F$ is the right adjoint functor to the forgetful functor from the category of finitely freely approximable functors to the category of all functors from~$\SNilp_N(S)$ to~$\Ab$.

\medskip

The following results allow to construct finitely freely approximable functors.

\begin{lemma}\label{lem:elappr2}
Let $\rho\colon F\to G$ be a morphism of finitely freely approximable functors from $\SNilp_N(S)$ to~$\Ab$. Then the cokernel~$\Coker(\rho)$ is a finitely freely approximable functor as well.
\end{lemma}
\begin{proof}
Indeed, taking (not necessarily filtered) colimits of abelian groups is right exact.
\end{proof}

\begin{proposition}\label{prop:elappr1}
Let $H$ be a functor from the category of rings to $\Ab$. Given an ideal~$I$ in a ring $R$, we put
$$
F_H(R,I):=\Ker\big(H(R)\to H(R/I)\big)\,.
$$
Suppose that the functor $H$ satisfies the following conditions:
\begin{itemize}
\item[(i)]
for any ring $R$ with a nilpotent ideal $I\subset R$, the corresponding homomorphism ${H(R)\to H(R/I)}$ is surjective;
\item[(ii)]
for any ring $R$ with a nilpotent ideal $I\subset R$, the group $F_H(R,I)$ is generated by the images of group homomorphisms
$$
F_H\big(R[[t]],(t)\big)\longrightarrow F_H(R,I)
$$
induced by homomorphisms of $R$-algebras $R[[t]]\to R$ that send $t$ to an element in $I$;
\item[(iii)]
for any ring $R$, the natural homomorphism ${\underset{A\subset R}\colim\, H(A)\longrightarrow H(R)}$
is an isomorphism, where the colimit is taken over finitely generated subrings~${A\subset R}$.
\end{itemize}
Then the functor
$$
F_H\;:\;\SNilp_N(S)\longrightarrow\Ab\,,\qquad (R,I)\longmapsto F_H(R,I)\,,
$$
is finitely freely approximable.
\end{proposition}
\begin{proof}
Let $(R,I)$ in $\SNilp_N(S)$ be such that $R$ is a finitely generated $S$-algebra. Put
\begin{equation}\label{eq:secondcolim}
\Gamma=\underset{(R',I')\twoheadrightarrow(R,I)}\colim\, H(R')\,,
\end{equation}
where the colimit is taken with respect to the following category $\Cc(R,I)$. Objects in~$\Cc(R,I)$ are arrows $f\colon (R',I')\twoheadrightarrow (R,I)$ in $\SNilp_N(S)$ such that $(R',I')$ is a finite free object in $\SNilp_N(S)$ and $f$ is surjective. Morphisms in $\Cc(R,I)$ from ${f_1\colon (R'_1,I'_1)\twoheadrightarrow (R,I)}$ to ${f_2\colon (R'_2,I'_2)\twoheadrightarrow (R,I)}$ are arrows $\varphi\colon (R_1',I_1')\to (R'_2,I'_2)$ in $\SNilp_N(S)$ such that $f_1=f_2\circ\varphi$. Let us prove that the natural homomorphism of groups
$$
\xi\;:\; \Gamma\to H(R)
$$
is an isomorphism.

Let $f\colon (R',I')\twoheadrightarrow (R,I)$ be an object in $\Cc(R,I)$. Put~${J=\Ker(f)\subset R'}$. The induced map $f\colon H(R')\to H(R)$ is equal to the composition
\begin{equation}\label{eq:compos}
H(R')\longrightarrow \Gamma\stackrel{\xi}\longrightarrow H(R)\,,
\end{equation}
where the first arrow is the natural morphism to the colimit. By definition, the homomorphism of rings $f\colon R'\twoheadrightarrow R$ is surjective. Since $f$ is a homomorphism of augmented $S$-algebras, the ideal $J$ is contained in the augmentation ideal~$I'$, whence $J$ is nilpotent. Consequently, by condition~(i) applied to $J\subset R'$, the map $f\colon H(R')\to H(R)$ is surjective. Taking into account composition~\eqref{eq:compos}, we obtain that $\xi$ is surjective as well.

Now we prove injectivity of $\xi$. One shows directly that $\Gamma$ coincides with the union of the images ${\rm Im}\big(H(R')\to \Gamma\big)$ over all objects $f\colon(R',I')\twoheadrightarrow (R,I)$ in~$\Cc(R,I)$. Thus to prove that $\xi$ is injective, it is enough to show the equality
$$
\Ker\big(H(R')\to \Gamma\big)=F_H(R',J)
$$
for any such object, where, as above, $J=\Ker(f)$. Composition~\eqref{eq:compos} yields an embedding ${\Ker\big(H(R')\to \Gamma\big)\subset F_H(R',J)}$. Let us show that there is also the converse embedding.

Consider an auxiliary object $(R'',I'')$ in $\Cc(R,I)$ defined as follows:
$$
R''=R'[[t]]/(I',t)^N\,,\qquad I''=(I',\bar t\,)\,,
$$
and the morphism to $(R,I)$ is the composition
$$
(R'',I'')\stackrel{g}\longrightarrow (R',I')\stackrel{f}\longrightarrow (R,I)\,,
$$
where $g$ is identical on $R'$ and sends $\bar t$ to zero. In particular, $g$ is a morphism in the category $\Cc(R,I)$. Let a subgroup $\Delta\subset H(R'')$ be the image of the map
$$
F_H\big(R'[[t]],(t)\big)\longrightarrow H(R'')
$$
induced by the quotient map $R'[[t]]\to R''$. Since the composition
$$
R'[[t]]\longrightarrow R''\stackrel{g}\longrightarrow R'
$$
is identical on $R'$ and sends $t$ to zero, we see that $g(\Delta)=0$ in $H(R')$.

Fix an element $x\in J$ and define another morphism $h\colon (R'',I'')\to (R',I')$ in~$\Cc(R,I)$ that is identical on $R'$ and sends $\bar t$ to $x$. Since $f,g$ are morphisms in $\Cc(R,I)$ and $g(\Delta)=0$, it follows from the definition of the colimit that the subgroup~${h(\Delta)\subset H(R')}$ is contained in $\Ker\big(H(R')\to \Gamma\big)$. Condition~(ii) applied to $J\subset R'$ implies that ${F_H(R',J)\subset \Ker\big(H(R')\to \Gamma\big)}$. As explained above, this finally proves that $\xi$ is an isomorphism.

Now let $(R,I)$ be an aritrary object in $\SNilp_N(S)$. There is a canonical isomorphism
$$
\underset{(R',I')\to(R,I)}\colim\,H(R')\stackrel{\sim}\longrightarrow
\underset{A\subset R}\colim\,\Big(\,\underset{(R',I')\twoheadrightarrow(A,A\cap I)}\colim\, H(R')\Big)\,,
$$
where the colimit in the left hand side is as in Definition~\ref{def:frapp}, $A$ runs over all finitely generated \mbox{$S$-subalgebras} in~$R$, and for each such $A$, we consider in the right hand side the colimit over the category~$\Cc(A,A\cap I)$. Indeed, for each morphism $(R',I')\to (R,I)$ from a finite free object $(R',I')$ in $\SNilp_N(S)$, one defines $A$ as the image of $R'$ in $R$.

By what was shown above, there is an isomorphism
$$
\underset{A\subset R}\colim\,\Big(\,\underset{(R',I')\twoheadrightarrow(A,A\cap I)}\colim\, H(R')\Big)\stackrel{\sim}\longrightarrow \underset{A\subset R}\colim\, H(A)\,,
$$
where the colimit in the right hand side is taken over all finitely generated \mbox{$S$-sub\-al\-ge\-bras}~${A\subset R}$.

Note that any finitely generated subring in $R$ is contained in a finitely generated $S$-algebra in~$R$ and any finitely generated $S$-subalgebra $A\subset R$ is a union of finitely generated subrings in $A$. Hence condition~(iii) applied to $R$ and to all finitely generated $S$-subalgebras in~$R$ implies that there is an isomorphism
$$
\underset{A\subset R}\colim\, H(A)\stackrel{\sim}\longrightarrow H(R)\,.
$$
Altogether this proves that the functor
\begin{equation}\label{eq:auxfunct}
\SNilp_N(S)\longrightarrow \Ab\,,\qquad (R,I)\longmapsto H(R)\,,
\end{equation}
is finitely freely approximable.

For any $(R,I)$ in $\SNilp_N(S)$, the $S$-algebra structure on $R$ defines an isomorphism ${H(R)\simeq F_H(R,I)\oplus H(S)}$, which is functorial with respect to $(R,I)$. Thus the functor~$F_H$ is finitely freely approximable, being a direct summand of the finitely freely approximable functor defined in formula~\eqref{eq:auxfunct}.
\end{proof}

\subsection{Construction of the Bloch map}\label{subsec:constrBloch}

As an application of Proposition~\ref{prop:elappr1}, we show that relative Milnor~\mbox{$K$-groups} (see Definition~\ref{def:relmil}) and relative modules of differential forms (see Definition~\ref{def:relforms}) define finitely freely approximable functors.

\begin{proposition}\label{prop:elaprKOmgega}
For any $n\geqslant 0$, the functors
$$
K_n^M\;:\;\SNilp_N(S)\longrightarrow\Ab\,,\qquad (R,I)\longmapsto K^{M}_{n}(R,I)\,,
$$
$$
\Omega^n\;:\;\SNilp_N(S)\longrightarrow\Ab\,,\qquad (R,I)\longmapsto \Omega^{n}_{R,I}\,,
$$
are finitely freely approximable.
\end{proposition}
\begin{proof}
Condition~(i) of Proposition~\ref{prop:elappr1} is satisfied for these functors trivially. Condition~(ii) of Proposition~\ref{prop:elappr1} is satisfied by Lemma~\ref{lem:generMilnorK} and Lemma~\ref{lem:generKahler}. Condition~(iii) of Proposition~\ref{prop:elappr1} is satisfied for Milnor $K$-groups and modules of differential forms, because they are given by finitely defined generators and relations. We finish the proof applying Proposition~\ref{prop:elappr1}.
\end{proof}

\begin{corollary}\label{cor:nonapprox}
\hspace{0cm}
\begin{itemize}
\item[(i)]
For any $n\geqslant 0$, the functor
$$
\SNilp_S\longrightarrow\Ab\,,\qquad (R,I)\longmapsto \Omega_{R,I}^n/d\,\Omega^{n-1}_{R,I}\,,
$$
is finitely freely approximable.
\item[(ii)]
Suppose that $N!$ is invertible in $S$. Then for any object $(R,I)$ in~$\SNilp_N(S)$, we have an isomorphism
$$
(d)^{fa}\;:\;\big(\Omega_{R,I}^n/d\,\Omega^{n-1}_{R,I}\big)^{fa}\simeq \Omega_{R,I}^n/d\,\Omega^{n-1}_{R,I}\stackrel{\sim}\longrightarrow \big((\Omega^{n+1}_{R,I})^{cl}\big)^{fa}\,.
$$
\end{itemize}
\end{corollary}
\begin{proof}
(i) This follows directly from Proposition~\ref{prop:elaprKOmgega} and Lemma~\ref{lem:elappr2}.

(ii) By part~(i), we have a homomorphism of groups
$$
(d)^{fa}\;:\;\big(\Omega_{R,I}^n/d\,\Omega^{n-1}_{R,I}\big)^{fa}\simeq \Omega_{R,I}^n/d\,\Omega^{n-1}_{R,I}\longrightarrow \big((\Omega^{n+1}_{R,I})^{cl}\big)^{fa}\,.
$$
We need to show that this is an isomorphism. By Lemma~\ref{lem:isomapprox}, it is enough to show this for $(R_{N,m},I_{N,m})$, $m\geqslant 0$. In this case, Proposition~\ref{prop:trivdR} claims that $H^{n}_{dR}(R_{N,m},I_{N,m})=0$, whence there is an isomorphism
\begin{equation}\label{eq:diffisom}
d\;:\;\Omega^n_{R_{N,m},I_{N,m}}/d\,\Omega^{n-1}_{R_{N,m},I_{N,m}}\stackrel{\sim}\longrightarrow (\Omega^{n+1}_{R_{N,m},I_{N,m}})^{cl}\,,
\end{equation}
which finishes the proof.
\end{proof}

In particular, it follows from Corollary~\ref{cor:nonapprox}(ii) that the functor $(\Omega^n_{R,I})^{cl}$ is not finitely freely approximable, because there are split nilpotent extensions with non-trivial relative de Rham cohomology (see the discussion at the end of Subsection~\ref{subsec:reldR}).

\medskip

Now we are ready to construct the Bloch map (see Definition~\ref{def:Blochmap}).

\begin{proof}[Proof of Theorem~\ref{the:Bloch}]
Let $R$, $I$, $N$, and $n$ be as in the theorem. Put~${S=R/I}$. Then~$(R,I)$ is a split nilpotent extension of $S$ of nilpotency degree~$N$ and~$N!$ is invertible in~$S$.
By Proposition~\ref{prop:elaprKOmgega} and Corollary~\ref{cor:nonapprox}, we obtain a homomorphism of groups
$$
(d\log)^{fa}\;:\;K_{n+1}^M(R,I)^{fa}\simeq K_{n+1}^M(R,I)\longrightarrow \big((\Omega^n_{R,I})^{cl}\big)^{fa}\simeq \Omega_{R,I}^n/d\,\Omega^{n-1}_{R,I}\,.
$$
Let us check that $(d\log)^{fa}$ satisfies the condition of Definition~\ref{def:Blochmap}, that is, that formula~\eqref{eq:defBloch} therein holds. Clearly, it is enough to show this for $(R_{N,m},I_{N,m})$, where $m\geqslant 0$. In this case, this is implied by the isomorphism~\eqref{eq:diffisom} and formula~\eqref{eq:Blochmap} from the proof of Lemma~\ref{lem:exact}. Thus~$(d\log)^{fa}$ is the Bloch map.
\end{proof}

In other words, we integrate $d\log$ by taking a (co)limit, in the spirit of calculus.

\begin{remark}
Let $S$ be a ring such that $N!$ is invertible in $S$. Our proof of Theorem~\ref{the:Bloch} implies that the restriction of the Bloch map to objects in $\SNilp_N(S)$ is the unique collection of homomorphisms ${K^M_{n+1}(R,I)\to \Omega^n_{R,I}/d\,\Omega^{n-1}_{R,I}}$ that are functorial with respect to $(R,I)$ in $\SNilp_N(S)$ and such that their composition with $d$ is $d\log$.
\end{remark}

It is important to mention that one can prove Theorem~\ref{the:Bloch} more directly, avoiding colimits. Nevertheless, the approach with finitely freely approximable functors will allow us further to prove also Theorem~\ref{thm:main}. Still let us sketch an explicit proof of Theorem~\ref{the:Bloch}.

As above, put $S=R/I$. Also, define a group
$$
{L_{n+1}=\Ker\big((R^*)^{\otimes (n+1)}\to (S^*)^{\otimes (n+1)}\big)}\,.
$$
First, one shows that formula~\eqref{eq:defBloch} from Definition~\ref{def:Blochmap} gives a homomorphism~${\widetilde{\rm B}\colon L_{n+1}\to\Omega^n_{R,I}/d\,\Omega^{n-1}_{R,I}}$ (actually, this holds in a non-split case as well). For this one uses that
$$
\log(1+x)\frac{d(1+y)}{1+y}+\log(1+y)\frac{d(1+x)}{1+x}=d\big(\log(1+x)\log(1+y)\big)\in d\,I
$$
for all elements $x,y\in I$ (cf. the end of the proof of Lemma~\ref{lem:generMilnorK}).

Then one needs to show that the homomorphism $\widetilde{\rm B}$ vanishes on the intersection~${L_{n+1}\cap {\rm St}_{n+1}(R)\subset (R^*)^{\otimes (n+1)}}$. Since $S$ splits out of $R$, one obtains an explicit system of generators of this intersection by projecting the Steinberg relations from $(R^*)^{\otimes (n+1)}$ to $L_{n+1}$ with respect to the decomposition 
$$
(R^*)^{\otimes (n+1)}\simeq L_{n+1}\oplus (S^*)^{\otimes (n+1)}\,.
$$
The vanishing of $\widetilde{\rm B}$ on these generators is reduced to the case $R=R_{N,m}$, when it holds, because the Bloch map exists in this case by Remark~\ref{rmk:Blochmap} and Proposition~\ref{prop:trivdR}.

As an example, consider the case $n=1$. Then the intersection ${L_2\cap {\rm St}_{2}(R)\subset (R^*)^{\otimes 2}}$ is generated by elements of type
$$
\zeta(a,x)=(a+x)\otimes (1-a-x)-a\otimes(1-a)\in (R^*)^{\otimes 2}\,,
$$
where $a,1-a\in S^*$ and $x\in I$. Clearly, the element $\zeta(a,x)$ is the image of the analogous element ${\zeta(a,\bar t\,)\in \big(R_{N,1}^*\big)^{\otimes 2}}$ under the map induced by the homomorphism of \mbox{$S$-algebras} ${R_{N,1}\to R}$ that sends $\bar t$ to $x$. The map $\widetilde{\rm B}$ vanishes on~$\zeta(a,\bar t\,)$, because the Bloch map~$\rm B$ exists for $R_{N,1}$ as its relative de Rham cohomology groups vanish by Lemma~\ref{lem:trivdR}.

\section{Milnor $K$-group of the ring of formal power series}\label{sect:powerseries}

Let $S$ be a ring such that $2$ is invertible in it and $S$ is weakly $5$-fold stable (see Definition~\ref{def:weakst}). In this section, we consider the Milnor $K$-group~${K_2^M\big(S[[t]]\big)}$. The main results are Proposition~\ref{prop:vanish} an Corollary~\ref{cor:vanish}.

\subsection{Filtrations on the Milnor $K$-group}\label{subsec:filtrMilnor}

One has a decreasing filtration on $S[[t]]^*$ given by the formula
$$
U_0:=S[[t]]^*=K_1^M\big(S[[t]]\big)\,,\qquad U_p:=1+t^pS[[t]]=K_1^M\big(S[[t]],(t^p)\big)\,,\qquad p\geqslant 1\,.
$$
This induces naturally a filtration on Milnor $K$-groups of $S[[t]]$. Namely, let the subgroup
\begin{equation}\label{eq:Vn}
V_p\subset K_2^M\big(S[[t]]\big)\,,\qquad p\geqslant 0\,,
\end{equation}
be generated by symbols $\{f,g\}$, where $f\in U_i$, $g\in U_j$ for $i,j\geqslant 0$ such that~${i+j\geqslant p}$. Also, put
\begin{equation}\label{eq:Wn}
W_0:=K_2^M\big(S[[t]]\big)\,,\qquad W_p:=K_2^M\big(S[[t]],(t^p)\big)\,,\qquad p\geqslant 1\,.
\end{equation}
It follows from Lemma~\ref{lem:generMilnorK} that there are embeddings
$$
V_p\supset W_p\supset V_{2p-1}\,,\qquad p\geqslant 0\,.
$$

\begin{proposition}\label{prop:vanish}
For any $p\geqslant 0$, there is an embedding
$$
p^{p-1}\cdot V_p\subset W_p
$$
of subgroups in $K_2^M\big(S[[t]]\big)$. In particular, if $p\geqslant 1$ is invertible in $S$, then there is an equality~${V_p=W_p}$.
\end{proposition}

Proposition~\ref{prop:vanish} is proved in Subsection~\ref{subsec:proofvanish}. Here is an important corollary of Proposition~\ref{prop:vanish}, which is, actually, its equivalent formulation.

\begin{corollary}\label{cor:vanish}
For all elements $a,b\in S$ and natural numbers $i,j,p\geqslant 0$ such that $i+j\geqslant p$, there is an equality
$$
p^{p-1}\cdot\{1+a\bar t^{\,i},1+b\bar t^{\,j}\}=0
$$
in $K_2^M\big(S[t]/(t^p)\big)$, where $\bar t$ denotes the image of $t$ in $S[t]/(t^p)$. In particular, if $p\geqslant 1$ is invertible in $S$, then for all elements $a,b\in S$ and natural numbers $i,j\geqslant 0$ such that $i+j\geqslant p$, there is a vanishing~${\{1+a\bar t^{\,i},1+b\bar t^{\,j}\}=0}$ in~${K_2^M\big(S[t]/(t^p)\big)}$.
\end{corollary}

%Let us mention that Proposition~\ref{prop:vanish} for the case when $R$ is a local $\Q$-algebra was proved essentially by Jagannathan~\cite[Theor.\,2]{Jag} using a result of %Kerz~\cite[Theor.\,6.1]{Kerz} as follows. Given an arbitrary finite subset $\Sigma\subset R$, there is a homomorphism from a regular local $\Q$-algebra to $R$ such that %$\Sigma$ is contained in its image. This implies that one can assume $R$ to be regular. Then it follows from the result of Kerz that the homomorphism $K_2^M\big(R[[t]]\big)\to %K_2\big(R[[t]]\big)$ is injective

We will use systematically an auxiliary $S$-algebra
$$
S'=S[[x]]
$$
and also the algebra $S'[[t]]=S[[x,t]]$ over $S[[t]]$. Let the subgroups
$$
U'_p\subset S'[[t]]^*\,,\qquad V'_p,W'_p\subset K_2^M\big(S'[[t]]\big)\,,\qquad p\geqslant 0\,,
$$
be defined similarly as~$U_p,V_p,W_p$ with $S$ being replaced by~$S'$. Define the homomorphisms of algebras over $S[[t]]$
$$
\theta_0\;:\; S'[[t]]\longrightarrow S[[t]]\,,\qquad f(x,t)\longmapsto f(0,t)\,,
$$
$$
\theta_q\;:\; S'[[t]]\longrightarrow S[[t]]\,,\qquad f(x,t)\longmapsto f(t^q,t)\,,
$$
where $q\geqslant 1$. In other words, $\theta_0$ is the quotient map to $S'[[t]]/(x)\simeq S[[t]]$.

The following lemma gives a way to make induction with respect to the indices in the filtration $V_p$, $p\geqslant 0$. The possibility of making such induction is the reason to introduce this filtration.

\begin{lemma}\label{lem:lift}
For all $p,q\geqslant 0$, there is an embedding
$$
(\theta_q-\theta_0)(V'_p)\subset V_{p+q}
$$
of subgroups in $K_2^M\big(S[[t]]\big)$, where we use the homomorphism of groups
\begin{equation}\label{eq:theta-theta}
\theta_q-\theta_0\;:\; K_2^M\big(S'[[t]]\big)\longrightarrow K_2^M\big(S[[t]]\big)\,.
\end{equation}
\end{lemma}
\begin{proof}
Consider a symbol $\{f,g\}\in V'_p$, where $f\in U'_i$, $g\in U'_j$, and ${i+j\geqslant p}$. Write
$$
f=\theta_0(f)\cdot \tilde f\,,\qquad g=\theta_0(g)\cdot \tilde g\,,
$$
where ${\theta_0(f),\theta_0(g)\in S[[t]]^*\subset S'[[t]]^*}$ and ${\tilde f, \tilde g\in 1+xS'[[t]]}$. We have equalities in~$S[[t]]^*$
$$
\theta_q(f)=\theta_0(f)\cdot\theta_q(\tilde f)\,,\qquad \theta_q(g)=\theta_0(g)\cdot\theta_q(\tilde g)\,.
$$
Thus there are equalities in $K_2^M\big(S[[t]]\big)$
$$
(\theta_q-\theta_0)\,\{f,g\}=\{\theta_q(f),\theta_q(g)\}-\{\theta_0(f),\theta_0(g)\}=
$$
$$
=\{\theta_0(f),\theta_q(\tilde g)\}+\{\theta_q(\tilde f),\theta_0(g)\}+\{\theta_q(\tilde f),\theta_q(\tilde g)\}\,.
$$
Clearly, ${\theta_0(f)\in U_i\subset U'_i}$ and ${\theta_0(g)\in U_j\subset U'_j}$. Therefore
$$
\tilde f\in U'_i\cap \big(1+xS'[[t]]\big)=1+xt^iS'[[t]]
$$
and, similarly, $\tilde g\in 1+xt^jS'[[t]]$. Hence we have ${\theta_q(\tilde f)\in U_{i+q}}$ and ${\theta_q(\tilde g)\in U_{j+q}}$, which finishes the proof.
\end{proof}

Lemma~\ref{lem:lift} implies immediately the following fact.

\begin{corollary}\label{cor:theta}
For all $p,q\geqslant 0$, the homomorphism in formula~\eqref{eq:theta-theta} induces a homomorphism
$$
K_2^M\big(S'[[t]]\big)/V'_p\longrightarrow K_2^M\big(S[[t]]\big)/V_{p+q}\,,
$$
which we denote also by $\theta_q-\theta_0$ for simplicity.
\end{corollary}

\subsection{Construction of symbols from differential forms}\label{subsec:constrsymbforms}

Let us give a way to construct elements in the quotients ${K_2^M\big(S[[t]]\big)/V_{p+1}}$, $p\geqslant 1$, from differential forms in $\Omega^1_S$.

Since $2$ is invertible in $S$ and the ring $S$ is weakly $5$-fold stable, by~\cite[Theor.\,2.9]{GorchinskiyOsipov2015Miln}, there is an isomorphism
\begin{equation}\label{eq:isomBloch2}
K_2^M\big(S[t]/(t^2),(\bar t\,)\big)\stackrel{\sim}\longrightarrow \Omega^1_S\,,
\end{equation}
which sends a symbol $\{1+ab\bar t,b\}$ to the differential form $adb$ for any $a\in S$ and $b\in S^*$ (see also Example~\ref{examp:Blochmap2}), where, as above, $\bar t$ denotes the image of $t$ in the quotient $S[t]/(t^2)$.

For any $p\geqslant 1$, define the homomorphism of $S$-algebras
$$
\lambda_p\;:\;S[[t]]\longrightarrow S[[t]]\,,\qquad f(t)\longmapsto f(t^p)\,.
$$
There are embeddings $\lambda_p(W_2)\subset W_{2p}\subset V_{2p}\subset V_{p+1}$. Thus the homomorphism
$$
\lambda_p\;:\; K_2^M\big(S[[t]]\big)\longrightarrow K_2^M\big(S[[t]]\big)
$$
induces a homomorphism
$$
K_2^M\big(S[[t]]\big)/{W_2}\longrightarrow K_2^M\big(S[[t]]\big)/{V_{p+1}}\,,
$$
which we denote also by $\lambda_p$ for simplicity.

Consider the composition
$$
\phi_p\;:\;\Omega^1_S\longrightarrow K^M_2\big(S[[t]]\big)/W_2\stackrel{\lambda_p}\longrightarrow K_2^M\big(S[[t]]\big)/V_{p+1}\,,
$$
where the first map is the inverse of the isomorphism~\eqref{eq:isomBloch2} followed by the embedding
$$
{K_2^M\big(S[t]/(t^2),(\bar t\,)\big)\subset K_2^M\big(S[t]/(t^2)\big)}\simeq K_2^M\big(S[[t]]\big)/W_2\,.
$$
In what follows, we denote similarly elements in~$K_2^M\big(S[[t]]\big)$ and their images under the quotient map ${K_2^M\big(S[[t]]\big)\to K_2^M\big(S[[t]]\big)/V_{p+1}}$ to ease notation. For any $a\in S$ and~${b\in S^*}$, we have an equality in~${K_2^M\big(S[[t]]\big)/V_{p+1}}$
\begin{equation}\label{eq:phi}
\phi_p(adb)=\{1+abt^p,b\}\,.
\end{equation}

To avoid confusion, we denote by $\phi'_p$ the map $\phi_p$ with $S'$ in place of $S$, that is, we define the maps
$$
\phi'_p\;:\;\Omega^1_{S'}\longrightarrow K_2^M\big(S'[[t]]\big)/V'_{p+1}\,,\qquad p\geqslant 1\,.
$$

\begin{lemma}\label{lem:expldif}
For all $p\geqslant 1$, $q\geqslant 0$, and $a,b\in S$, there is an equality
$$
\big((\theta_q-\theta_0)\circ \phi'_p\big)(axdb)=\phi_{p+q}(adb)
$$
in $K_2^M\big(S[[t]]\big)/V_{p+q+1}$, where we use the homomorphism of groups
$$
\theta_q-\theta_0\;:\; K_2^M\big(S'[[t]]\big)/V'_{p+1}\longrightarrow K_2^M\big(S[[t]]\big)/V_{p+q+1}
$$
from Corollary~\ref{cor:theta}.
\end{lemma}
\begin{proof}
Since the ring $S$ is weakly $5$-fold stable and, in particular, is weakly \mbox{$2$-fold} stable,~$S$ is generated additively by invertible elements. Hence it is enough to prove the lemma in the case when $b$ is invertible. Then, by formula~\eqref{eq:phi}, there are equalities
$$
\big((\theta_q-\theta_0)\circ \phi'_p\big)(axdb)=(\theta_q-\theta_0)\,\{1+axbt^{p},b\}=\{1+ab t^{p+q},b\}=\phi_{p+q}(adb)
$$
in $K_2^M\big(S[[t]]\big)/V_{p+q+1}$.
\end{proof}

\subsection{Key result}\label{subsec:key}

Recall the following lemma, which was proved by Morrow~\cite[Lem.\,3.6]{Morrow} using a method which goes back to Nesterenko and Suslin~\cite[Lem.\,3.2]{NesterenkoSuslin} (see also Lemma 2.2 from the paper of Kerz~\cite{Kerz}).

\begin{lemma}\label{lem:sign}
Let $R$ be a weakly $5$-fold stable ring. Then for all elements $r, s\in R^{*}$, there is an equality $\{r,s\}=-\{s,r\}$ in~$K^{M}_{2}(R)$.
\end{lemma}

Here is a key result to prove further Proposition~\ref{prop:vanish}.

\begin{proposition}\label{prop:key}
For all elements $a,b\in S$ and natural numbers $i,j\geqslant 1$, there is an equality
\begin{equation}\label{eq:key}
(i+j)\{1+at^i,1+bt^j\}=\phi_{i+j}(iadb-jbda)
\end{equation}
in $K_2^M\big(S[[t]]\big)/V_{i+j+1}$.
\end{proposition}
\begin{proof}
The proof is by induction on the sum $i+j$. First we consider the base of the induction, that is, the case $i+j=2$ or, equivalently, $i=j=1$. The proof repeats that of~\cite[Lemma\,3.3]{GorchinskiyOsipov2015Miln} with a minor refinement.  We need to prove that for any pair~$(a,b)$ of elements in $S$, there is an equality
\begin{equation}\label{eq:key3}
2\{1+at,1+bt\}=\phi_2(adb-bda)
\end{equation}
in $K_2^M\big(S[[t]]\big)/V_3$. Note that both sides of~\eqref{eq:key3} are linear in $a$ and $b$. Since $S$ is generated additively by invertible elements, we may assume that $a,b\in S^*$.

Moreover, since $S$ is weakly $4$-fold stable, there is $c\in S^*$ such that the elements
$$
\frac{a+b}{2}+c\,,\qquad a+c\,,\qquad b+c
$$
are invertible in~$S$. By bilinearity, it is enough to prove~\eqref{eq:key3} for the pairs ${(a+c,b+c)}$, ${(a,c)}$, $(c,b)$, and $(c,c)$. Note that each of these pairs satisfies the following condition: its both terms and their sum are invertible. Thus we may assume that~${a,b,a+b\in S^*}$.

We have the Steinberg relations in $K_2^M\big(S[[t]]\big)$
$$
\left\{\frac{b}{a+b}(1+at),\frac{a}{a+b}(1-bt)\right\}=\left\{\frac{b}{a+b},\frac{a}{a+b}\right\}=0\,.
$$
Subtracting the second symbol from the first one, we obtain an equality in~$K_2^M\big(S[[t]]\big)$
\begin{equation}\label{eq:denis1}
\left\{\frac{b}{a+b},1-bt\right\}+\left\{1+at,\frac{a}{a+b}\right\}+\{1+at,1-bt\}=0\,.
\end{equation}
Applying the automorphism of the $S$-algebra $S[[t]]$ that sends a series $f(t)$ to~$f(-t)$, we get
\begin{equation}\label{eq:denis2}
\left\{\frac{b}{a+b},1+bt\right\}+\left\{1-at,\frac{a}{a+b}\right\}+\{1-at,1+bt\}=0\,.
\end{equation}
Besides, the equalities
$$
{\{1+at,1-bt\}=\{1-at,1+bt\}=-\{1+at,1+bt\}}
$$
hold in the quotient $K_2^M\big(S[[t]]\big)/V_3$. Hence, taking the sum of~\eqref{eq:denis1} and~\eqref{eq:denis2}, we obtain an equality in~$K_2^M\big(S[[t]]\big)/V_3$
$$
\left\{\frac{b}{a+b},1-b^2t^2\right\}+\left\{1-a^2t^2,\frac{a}{a+b}\right\}-2\{1+at,1+bt\}=0\,.
$$
Since the ring $S[[t]]$ is weakly $5$-fold stable, applying Lemma~\ref{lem:sign}, we get equalities in~${K_2^M\big(S[[t]]\big)/V_3}$
$$
2\{1+at,1+bt\}=-\left\{1-b^2t^2,\frac{b}{a+b}\right\}+\left\{1-a^2t^2,\frac{a}{a+b}\right\}=
$$
$$
=-\{1-b^2t^2,b\}+\{1-b^2t^2,a+b\}+\{1-a^2t^2,a\}-\{1-a^2t^2,a+b\}=
$$
$$
=-\{1-b^2t^2,b\}+\{1-a^2t^2,a\}+\{1+(a^2-b^2)t^2,a+b\}\,.
$$
By formula~\eqref{eq:phi} from Subsection~\ref{subsec:constrsymbforms}, the latter expression is equal to
$$
\phi_2\big(bdb-ada+(a-b)d(a+b)\big)=\phi_2(adb-bda)\,,
$$
which proves the base of the induction.

Let us make the induction step to arbitrary $i+j$. First suppose that~${i=j}$. There are embeddings $\lambda_i(V_3)\subset V_{3i}\subset V_{2i+1}$. Hence the homomorphism~$\lambda_i$ induces a homomorphism
$$
K_2^M\big(S[[t]]\big)/V_3 \longrightarrow K_2^M\big(S[[t]]\big)/V_{2i+1}\,,
$$
which we denote also by $\lambda_i$ for simplicity.

Let us apply this homomorphism to equality~\eqref{eq:key3}. Note that the composition
$$
\Omega^1_S\stackrel{\phi_2}\longrightarrow K_2^M\big(S[[t]]\big)/V_3 \stackrel{\lambda_i}\longrightarrow K_2^M\big(S[[t]]\big)/V_{2i+1}
$$
is equal to $\phi_{2i}$, because~${\lambda_i\circ\lambda_2=\lambda_{2i}}$. Hence we obtain the equality
$$
2\{1+at^i,1+bt^i\}=\phi_{2i}(adb-bda)
$$
in $K_2^M\big(S[[t]]\big)/V_{2i+1}$. This gives the statement in the case $i=j$.

Now suppose that $i\ne j$. We can assume that $i<j$ (the other case is done similarly or one can use Lemma~\ref{lem:sign}). Since $S$ is weakly $5$-fold stable, the same holds for the ring~$S'=S[[x]]$. Apply the induction hypothesis to the ring $S'$ in place of $S$, to the elements $a,bx\in S'$, and to the exponents $i$, $j-i$. We obtain the equality
$$
j\{1+at^i,1+bxt^{j-i}\}=\phi'_{j}\big(iad(bx)-(j-i)bxda\big)
$$
in $K_2^M\big(S'[[t]]\big)/V'_{j+1}$.

Applying the map
$$
\theta_i-\theta_0\;:\;K_2^M\big(S'[[t]]\big)/V'_{j+1}\longrightarrow K_2^M\big(S[[t]]\big)/V_{i+j+1}
$$
from Corollary~\ref{cor:theta}, we get the equality
\begin{equation}\label{eq:keyaux}
j\{1+at^i,1+bt^j\}=\big((\theta_i-\theta_0)\circ\phi'_{j}\big)\big(iad(bx)-(j-i)bxda\big)
\end{equation}
in $K_2^M\big(S[[t]]\big)/V_{i+j+1}$.

Let us compare formula~\eqref{eq:keyaux} with the needed formula~\eqref{eq:key}. By Lemma~\ref{lem:expldif}, the right hand side of formula~\eqref{eq:key} is equal to
$$
\phi_{i+j}(iadb-jbda)=\big((\theta_i-\theta_0)\circ\phi'_{j}\big)(iaxdb-jbxda)\,.
$$
Besides, there are equalities in $\Omega^1_S$
$$
\big(iad(bx)-(j-i)bxda\big)-(iaxdb-jbxda)=iabdx+ibxda=ibd(ax)=ibd(1+ax)\,.
$$
Therefore the difference between the right hand side of formula~\eqref{eq:keyaux} and the right hand side of formula~\eqref{eq:key} is equal to
$$
\big((\theta_i-\theta_0)\circ\phi'_j\big)\big(ibd(1+ax)\big)=i(\theta_i-\theta_0)\,\{1+b(1+ax)t^j,1+ax\}=
$$
$$
=i\{1+bt^j+abt^{i+j},1+at^i\}=i\{1+bt^j,1+at^i\}=-i\{1+at^i,1+bt^j\}\,,
$$
where all equalities are in~${K_2^M\big(S[[t]]\big)/V_{i+j+1}}$, we use formula~\eqref{eq:phi} from Subsection~\ref{subsec:constrsymbforms} for the first equality and we use Lemma~\ref{lem:sign} applied to the weakly $5$-stable ring $S[[t]]$ for the last equality. The latter expression is equal to the difference between the left hand side of formula~\eqref{eq:keyaux} and the left hand side of the needed formula~\eqref{eq:key}. This proves the proposition.
\end{proof}

\subsection{Proof of Proposition~\ref{prop:vanish}}\label{subsec:proofvanish}

First we deduce a useful corollary from Propostion~\ref{prop:key}.

\begin{corollary}\label{cor:key}
For any $p\geqslant 0$, there is an embedding
$$
p\cdot V_p\subset W_p+V_{p+1}
$$
of subgroups in $K_2^M\big(S[[t]]\big)$.
\end{corollary}
\begin{proof}
The case $p=0$ is trivial. For $p\geqslant 1$, note that the group~${V_p}$ is generated by symbols of type $\{1+at^p,b\}$, $\{b,1+at^p\}$, where $a\in S$, $b\in S^*$, by symbols of type $\{1+at^i,1+bt^j\}$, where $i,j\geqslant 1$, $i+j=p$, $a,b\in S$, and by elements of the subgroup~$V_{p+1}$.

Clearly, symbols of the first type belong to $W_p$. By Proposition~\ref{prop:key}, symbols of the second type satisfy the condition
$$
p\{1+at^i,1+bt^j\}\in \phi_{p}(iadb-jbda)+V_{p+1}\,.
$$
Furthermore, we claim that ${\phi_{p}(iadb-jbda)\in W_{p}}$. Indeed, since $S$ is generated additively by invertible elements, we may assume that $a,b\in S^*$ and use formula~\eqref{eq:phi} from Subsection~\ref{subsec:constrsymbforms}. Hence, symbols of the second type multiplied by~$p$ belong to~${W_p+V_{p+1}}$.
\end{proof}

Now we are ready to prove Proposition~\ref{prop:vanish}.

\begin{proof}[Proof of Proposition~\ref{prop:vanish}]
Let us prove that for any $q$ with ${0\leqslant q\leqslant p-2}$, there is an embedding
\begin{equation}\label{eq:induct}
p\cdot V_{p+q}\subset W_p+V_{p+q+1}\,.
\end{equation}
of subgroups in $K_2^M\big(S[[t]]\big)$.

Since $S'$ is weakly $5$-fold stable, by Corollary~\ref{cor:key} applied with $S'$ in place of~$S$, we have
\begin{equation}\label{eq:inclus}
p\cdot V'_p\subset W'_p+V'_{p+1}\,.
\end{equation}
Since $\theta_q(t^p)=t^p$ and $\theta_0(t^p)=t^p$, we see that ${\theta_q(W'_p)\subset W_p}$ and ${\theta_0(W'_p)\subset W_p}$. In particular, we have
$$
(\theta_q-\theta_0)(W'_p)\subset W_p\,.
$$
By Lemma~\ref{lem:lift}, there are embeddings
$$
(\theta_q-\theta_0)(V'_p)\subset V_{p+q}\,,\qquad  (\theta_q-\theta_0)(V'_{p+1})\subset V_{p+q+1}\,.
$$
Moreover, we claim that there is an equality
\begin{equation}\label{eq:equalityV}
(\theta_q-\theta_0)(V'_p)=V_{p+q}\,.
\end{equation}
Indeed, the group $V_{p+q}$ is generated by symbols of type $\{1+at^i,1+bt^j\}$, where $i,j\geqslant 0$, $i+j\geqslant p+q$, and $a,b\in S$. Since $q<p$, either $i$ or $j$ is greater than~$q$. We may assume that $i>q$. Then we have
$$
(\theta_q-\theta_0)\,\{1+axt^{i-q},1+bt^j\}=\{1+at^i,1+bt^j\}\,,
$$
where $\{1+axt^{i-q},1+bt^j\}\in V'_{p}$. This proves equality~\eqref{eq:equalityV}.

Now, applying the map $\theta_q-\theta_0$ to formula~\eqref{eq:inclus}, we obtain formula~\eqref{eq:induct}. Using formula~\eqref{eq:induct} repeatedly, we get the embeddings
$$
p^{p-1}\cdot V_p\subset p^{p-2}\cdot (W_p+V_{p+1})\subset W_p+p^{p-2}\cdot V_{p+1}\subset\ldots\subset W_p+p\cdot V_{2p-2}\subset W_p+V_{2p-1}\,.
$$
Since $V_{2p-1}\subset W_p$, this proves the proposition.
\end{proof}

Note that in the proof of Proposition~\ref{prop:vanish} we use only Corollary~\ref{cor:key} and not Proposition~\ref{prop:key} itself. However, it is not clear to us whether it is possible to prove Corollary~\ref{cor:key} directly by induction. This is why we have replaced it with a more precise and stronger statement, namely, with Proposition~\ref{prop:key}, which admits an inductive proof.

\section{Proof of Theorem~\ref{thm:main}}\label{sec:proofmain}

\subsection{Reduction lemma}\label{subsec:redlemma}

The following lemma is our main tool to make reductions when proving that the Bloch map (see Definition~\ref{def:Blochmap}) is an isomorphism.

\begin{lemma}\label{lem:embed}
Let~$J\subset I$ be two nilpotent ideals in a ring $R$. Put $R'=R/J$, $I'=I/J$. Suppose that for a natural number $n\geqslant 0$, there are equalities
$$
H^n_{dR}(R,I)=H^n_{dR}(R',I')=H^{n-1}_{dR}(R',I')=0\,.
$$
Then the following holds:
\begin{itemize}
\item[(i)]
the Bloch maps
$$
{\rm B}_{R,I}\;:\;K_{n+1}^M(R,I)\longrightarrow \Omega^n_{R,I}/d\,\Omega^{n-1}_{R,I}\,,
$$
$$
\B_{R',I'}\;:\;K_{n+1}^M(R',I')\longrightarrow \Omega^n_{R',I'}/d\,\Omega^{n-1}_{R',I'}\,,
$$
$$
\B_{R,J}\;:\;K_{n+1}^M(R,J)\longrightarrow \Omega^n_{R,J}/d\,\Omega^{n-1}_{R,J}
$$
exist;
\item[(ii)]
if, in addition, any element in $R$ is a sum of invertible elements, then the Bloch map~$\B_{R,I}$ is an isomorphism if and only if both~$\B_{R',I'}$ and~$\B_{R,J}$ are isomorphisms.
\end{itemize}
\end{lemma}
\begin{proof}
(i)
Since $H_{dR}^n(R,I)=H_{dR}^n(R',I')=0$, by Remark~\ref{rmk:Blochmap}, the Bloch maps~$\B_{R,I}$ and $\B_{R',I'}$ exist. Using also that $H^{n-1}_{dR}(R',I')=0$ and the exact sequence~\eqref{eq:longseqcohom} of relative de Rham cohomology from Subsection~\ref{subsec:basicdiffforms}, we obtain $H_{dR}^n(R,J)=0$. Thus, applying Remark~\ref{rmk:Blochmap} again, we see that the Bloch map $\B_{R,J}$ exists as well.

(ii)
The condition $H^{n-1}_{dR}(R',I')=0$ and the exact sequence~\eqref{eq:longseqforms} from Subsection~\ref{subsec:basicdiffforms} imply that we have an exact sequence
$$
0\longrightarrow \Omega^{n}_{R,J}/d\,\Omega^{n-1}_{R,J}\longrightarrow\Omega^{n}_{R,I}/
d\,\Omega^{n-1}_{R,I} \longrightarrow\Omega^{n}_{R',I'}/d\,\Omega^{n-1}_{R',I'}\longrightarrow 0\,.
$$
Furthermore, there is a commutative diagram with exact raws
$$
\begin{CD}
0@>>> K^M_{n+1}(R,J)@>>>  K^M_{n+1}(R,I)@>>>  K^M_{n+1}(R',I')@>>>0 \\
@.  @VV{\B}_{R,J}V @VV{\B}_{R,I}V @VV{\B}_{R',I'}V\\
0@>>> \Omega^n_{R,J}/d\,\Omega^{n-1}_{R,J} @>>> \Omega^n_{R,I}/d\,\Omega^{n-1}_{R,I} @>>> \Omega^n_{R',I'}/d\,\Omega^{n-1}_{R',I'}@>>>0
\end{CD}
$$
Since $R$ is generated additively by invertible elements, the same holds for its quotient~${R'}$. So, by Lemma~\ref{lem:surjBlochmap}, the Bloch maps in the diagram are surjective. Therefore, there is an exact sequence
$$
0\longrightarrow \Ker(\B_{R,J})\longrightarrow \Ker(\B_{R,I})\longrightarrow \Ker(\B_{R',I'})\longrightarrow 0\,,
$$
which finishes the proof.
\end{proof}

\subsection{Special case of the main result}\label{subsec:partialmain}

Let $S$ be a ring such that $N!$ is invertible in~it for a natural number $N\geqslant 1$ and $S$ is weakly $5$-fold stable (see Definition~\ref{def:weakst}).

For short, denote the \mbox{$S$-algebra}~$R_{N,1}$ (see Definition~\ref{def:elem}) just by $R_N$, that is, put
$$
R_N:=S[t]/(t^N)\,.
$$
As above, $\bar t$ denotes the image of $t$ in $R_N$.

In this subsection, we prove Theorem~\ref{thm:main} for $(\bar t)\subset R_N$. The case $N=1$ is trivial, so we assume that $N\geqslant 2$.

Define the homomorphism of \mbox{$S$-algebras}
$$
\sigma\;:\; R_2\longrightarrow R_N\,,\qquad \bar t\longmapsto \bar t^{\,N-1}\,.
$$
Appying Corollary~\ref{cor:vanish} of Proposition~\ref{prop:vanish}, we prove the following useful fact.

\begin{proposition}\label{prop:k2}
For any natural number $n\geqslant 0$, the homomorphism of groups ${\sigma\colon K_{n+1}^M(R_2)\to K_{n+1}^M(R_N)}$ restricts to a surjective homomorphism
$$
\sigma\;:\;K^{M}_{n+1}\big(R_2,(\bar t\,)\big)\longrightarrow K^{M}_{n+1}\big(R_{N},(\bar t^{\,N-1})\big)\,.
$$
\end{proposition}
\begin{proof}
Since the ring $S$ is weakly $5$-fold stable, the same holds for the ring~$R_N$, so we may apply Lemma~\ref{lem:sign} to $R_N$. Combining this with Lemma~\ref{lem:generMilnorK}, we see that the group $K^{M}_{n+1}\big(R_{N},(\bar t^{\,N-1})\big)$ is generated by elements of type ${\{1+a\bar t^{\,N-1},r_{1},\ldots,r_{n}\}}$, where $a\in S$ and $r_{1},\ldots,r_{n}\in R_N^*$. Decomposing invertible elements of $R_N$, we can assume that for each~$i$, $1\leqslant i\leqslant n$, we have either $r_i\in S^*$ or $r_i=1+b\bar t^{\,j}$ for some $b\in S$ and~${j\geqslant 1}$.

If all the elements $r_i$ are from $S^*$, then the symbol ${\{1+a\bar t^{\,N-1},r_{1},\ldots,r_{n}\}}$ is clearly in the image of the map in question. Otherwise, using Lemma~\ref{lem:sign} again, we may assume that $r_1=1+b\bar t^{\,j}$, where $b\in S$ and $j\geqslant 1$. Now, by Corollary~\ref{cor:vanish}, the symbol
$$
\{1+a\bar t^{\,N-1},r_1\}=\{1+a\bar t^{\,N-1},1+b\bar t^{\,j}\}
$$ vanishes in $K_2^M(R_N)$, because $N$ is invertible in $S$. Hence the symbol ${\{1+a\bar t^{\,N-1},r_{1},\ldots,r_{n}\}}$ vanishes as well, which finishes the proof.
\end{proof}

%Note that in the proof of Proposition~\ref{prop:k2} we apply Proposition~\ref{prop:vanish} only for the case $p=N-1$. %However, the inductive proof of Proposition~\ref{prop:vanish} requires the full statement for arbitrary $p$ and $q$.

We will also need the following simple statement.

\begin{lemma}\label{lem:isomdiff}
For any $n\geqslant 0$, the morphism of $S$-modules $\sigma\colon \Omega^n_{R_2}\to \Omega^n_{R_N}$ restricts to an isomorphism of $S$-modules
$$
\sigma\;:\;\Omega^n_{R_2,(\bar t\,)}\stackrel{\sim}\longrightarrow \Omega^n_{R_N,(\bar t^{N-1})}\,.
$$
\end{lemma}
\begin{proof}
Let us describe the $S$-module $\Omega^{n}_{R_N}$ explicitly. We have an isomorphism
$$
\Omega^n_{S[t]}\simeq \Omega^{n}_{S}[t]\oplus \Omega^{n-1}_S[t]\wedge dt\,.
$$
Also, there is an equality ${d(t^{N})=Nt^{N-1}dt}$. Combining this with the fact that~$N$ is invertible in~$S$ and with Lemma~\ref{lem:generKahler} applied to the ring $S[t]$ and the ideal $(t^N)\subset S[t]$, we obtain an isomorphism of $S$-modules
$$
\Omega^{n}_{R_N}\simeq(\Omega^{n}_{S}\oplus\Omega^{n}_{S}\cdot \bar t\oplus\ldots\oplus\Omega^{n}_{S}\cdot \bar t^{\,N-1})\oplus (\Omega^{n-1}_S\wedge d\bar t\oplus \Omega^{n-1}_S\wedge \bar t d\bar t\oplus\ldots\oplus \Omega^{n-1}_S\wedge \bar t^{\,N-2}d\bar t\,)\,.
$$
Applying Lemma~\ref{lem:generKahler} to the ring $R_N$ and the ideal $\bar t^{\,N-1}$, we get an isomorphism
\begin{equation}\label{eq:isom1}
\Omega^n_{R_N,(\bar t^{N-1})}\simeq \Omega_S^n\cdot\bar t^{\,N-1}\oplus \Omega^{n-1}_S\wedge \bar t^{\,N-2}d\bar t\,.
\end{equation}
When $N=2$, this gives an isomorphism (cf. Example~\ref{examp:Blochmap2})
\begin{equation}\label{eq:isom2}
\Omega^{n}_{R_2,(\bar t\,)}\simeq \Omega_S^n\cdot \bar t\oplus \Omega^{n-1}_S\wedge d\bar t\,.
\end{equation}
Clearly, $\sigma$ induces an isomorphism between the right hand sides of~\eqref{eq:isom2} and~\eqref{eq:isom1}.
\end{proof}

Now, combining Proposition~\ref{prop:k2}, Lemma~\ref{lem:isomdiff}, and the main result of~\cite{GorchinskiyOsipov2015Miln}, we prove the following auxiliary special case of Theorem~\ref{thm:main}.

\begin{proposition}\label{prop:m=1}
For any $n\geqslant 0$, the Bloch map for ${(\bar t\,)\subset R_N}$ is an isomorphism, that is, we have an isomorphism
$$
\B\;:\; K^{M}_{n+1}\big(R_{N},(\bar t\,)\big)\stackrel{\sim}\longrightarrow\Omega^n_{R_N,(\bar t\,)}/d\,\Omega^{n-1}_{R_N,(\bar t\,)}\,.
$$
\end{proposition}
\begin{proof}
The proof is by induction on $N$. The base of the induction, namely, the case~$N=2$, is~\cite[Theor.\,2.9]{GorchinskiyOsipov2015Miln}, which holds for $S$, because $2$ is invertible in it and $S$ is weakly $5$-fold stable.

Let us make an induction step from ${N-1}$ to $N$. We will apply Lemma~\ref{lem:embed} to the ideals
$$
(\bar t^{\,N-1})\subset (\bar t\,)\subset R_N\,.
$$
By definition, we have an isomorphism $R_{N-1}\simeq R_N/(\bar t^{\,N-1})$. Since $N!$ is invertible in $S$, by Lemma~\ref{lem:trivdR}, we have the vanishing of relative de Rham cohomology for $(\bar t\,)\subset R_N$ and for~$(\bar t\,)\subset R_{N-1}$. Hence by Lemma~\ref{lem:embed}(i), the Bloch map exists for~${(\bar t^{\,N-1})\subset R_N}$.

For any $n\geqslant 0$, we have the following commutative diagram:
$$
\begin{CD}
K^{M}_{n+1}\big(R_2,(\bar t\,)\big)@>{\B}>> \Omega^n_{R_2,(\bar t\,)}/d\,\Omega^{n-1}_{R_2,(\bar t\,)} \\
@VV\sigma V @VV\sigma V \\
K^{M}_{n+1}\big(R_{N},(\bar t^{\,N-1})\big)@>{\B}>> \Omega^n_{R_N,(\bar t^{N-1})}/d\,\Omega^{n-1}_{R_N,(\bar t^{N-1})}
\end{CD}
$$
The left vertical map is surjective by Proposition~\ref{prop:k2}. The right vertical map is an isomorphism by Lemma~\ref{lem:isomdiff}. The top horizontal map is an isomorphism by~\cite[Theor.\,2.9]{GorchinskiyOsipov2015Miln}. Altogether this implies that the left vertical map and the bottom horizontal map are isomorphisms as well. So, the Bloch map is an isomorphism for $(\bar t^{\,N-1})\subset R_N$ (notice that this is an instance when the Bloch map exists and is an isomorphism in a non-split case).

By the induction hypothesis, the Bloch map is also an isomorphism for~${(\bar t\,)\subset R_{N-1}}$. Since the ring $S$ is weakly $5$-fold stable and, in particular, is weakly $2$-fold stable,~$S$ is generated additively by invertible elements. Hence by Lemma~\ref{lem:embed}(ii), the Bloch map is an isomorphism for ${(\bar t\,)\subset R_N}$.
\end{proof}

Note that in the proof of Proposition~\ref{prop:k2} (respectively, of Lemma~\ref{lem:isomdiff}) we used only that $2N$ (respectively, $N$) is invertible in $S$. However, in the proof of Proposition~\ref{prop:m=1} we do use the invertibility of~$N!$ in~$S$.

\subsection{Proof of the main result}\label{subsec:proofmain}

Now we are ready to prove Theorem~\ref{thm:main}.

\begin{proof}[Proof of Theorem~\ref{thm:main}]
Let $R$, $I$, $N$, and $n$ be as in the theorem. Put~${S=R/I}$. Then~$(R,I)$ is a split nilpotent extension of $S$ of nilpotency degree~$N$ (see Definition~\ref{def:splitnilp}), we have that~$N!$ is invertible in~$S$, and the ring~$S$ is weakly $5$-fold stable. The Bloch map is a morphism from the functor
$$
K_{n+1}^M\;:\;\SNilp_N(S)\longrightarrow\Ab\,,\qquad (R,I)\longmapsto K_{n+1}^M(R,I)
$$
to the functor
$$
\Omega^n/d\,\Omega^{n-1}\;:\;\SNilp_N(S)\longrightarrow\Ab\,,\qquad (R,I)\longmapsto \Omega^n_{R,I}/d\,\Omega^{n-1}_{R,I}\,.
$$
Thus the theorem is equivalent to the fact that this morphism of functors is an isomorphism.

By Proposition~\ref{prop:elaprKOmgega} and Corollary~\ref{cor:nonapprox}(i), respectively, the functors $K_{n+1}^M$ and~$\Omega^n/d\,\Omega^{n-1}$ are finitely freely approximable (see Definition~\ref{def:frapp}). Hence by Lemma~\ref{lem:isomapprox}, it is enough to prove the theorem for $I_{N,m}\subset R_{N,m}$ (see Definition~\ref{def:elem}), where $m\geqslant 1$.

We proceed by induction on $m$. The base of the induction $m=1$ is Proposition~\ref{prop:m=1}. Let us make an induction step from $m-1$ to $m$.

Let $J\subset R_{N,m-1}[t_m]/(t_m^N)$ be the kernel of the surjective homomorphism of algebras over $R_{N,m-1}$
$$
R_{N,m-1}[t_m]/(t_m^N)\longrightarrow R_{N,m}\,,\qquad t_m\longmapsto \bar t_m\,.
$$
Explicitly, we have
$$
J=(\bar t_1,\ldots,\bar t_{m-1})^{N-1}\cdot \bar t_m+(\bar t_1,\ldots,\bar t_{m-1})^{N-2}\cdot \bar t_m^{\,2}+\ldots (\bar t_1,\ldots,\bar t_{m-1})\cdot \bar t^{\,N-1}_m\,.
$$
In particular, $J\subset (\bar t_m)$. Let us apply Lemma~\ref{lem:embed} to the ideals
$$
J\subset (\bar t_m)\subset {R_{N,m-1}[t_m]/(t_m^N)}\,.
$$
By Lemma~\ref{lem:trivdR} with~$S$ replaced by~$R_{N,m-1}$, we have the vanishing of relative de Rham cohomology for ${(\bar t_m)\subset R_{N,m-1}[t_m]/(t_m^N)}$. Furthermore, the preimage of the ideal $J$ in~$R_{N,m-1}[t_m]$ contains $t^N_m$ and is generated by monomials in $t_m$. Hence by Lemma~\ref{lem:trivdR} again, relative de Rham cohomology groups are trivial for ${(\bar t_m)\subset R_{N,m}}$.

Since $N!$ is invertible in $S$ and the ring $S$ is weakly $5$-fold stable, the same holds for the ring $R_{N,m-1}$.
Thus by Proposition~\ref{prop:m=1} with $S$ replaced by~$R_{N,m-1}$, the Bloch map is an isomorphism for ${(\bar t_m)\subset R_{N,m-1}[t_m]/(t_m^N)}$. Hence by Lemma~\ref{lem:embed}(ii), the Bloch map is an isomorphism for ${(\bar t_m)\subset R_{N,m}}$ (and also for $J\subset R_{N,m-1}[t_m]/(t_m^N)$, though we are not using this fact).

Now we apply Lemma~\ref{lem:embed} to the ideals
$$
(\bar t_m)\subset I_{N,m}\subset R_{N,m}\,.
$$
By definition, we have isomorphisms
$$
R_{N,m-1}\simeq R_{N,m}/(\bar t_m)\,,\qquad I_{N,m-1}\simeq I_{N,m}/(\bar t_m)\,.
$$
Proposition~\ref{prop:trivdR} claims the vanishing of relative de Rham cohomology for~${I_{N,m}\subset R_{N,m}}$ and for~${I_{N,m-1}\subset R_{N,m-1}}$. By what was shown above, the Bloch map is an isomorphism for $(\bar t_m)\subset R_{N,m}$. In addition, by the induction hypothesis, the Bloch map is an isomorphism for $I_{N,m-1}\subset R_{N,m-1}$. Hence by Lemma~\ref{lem:embed}(ii), the Bloch map is an isomorphism for~$I_{N,m}\subset R_{N,m}$, which finishes the proof.
\end{proof}

\subsection{Non-existence of the Bloch map in a non-split case}\label{subsec:nonexistBloch}

Finally, using Theorem~\ref{thm:main}, we show that the Bloch map does not exist in general in a non-split case.

\begin{proposition}\label{prop:nonBloch}
Let $S$ be a ring such that $N!$ is invertible in $S$ for a natural number $N\geqslant 1$ and $S$ is weakly $5$-fold stable. Let $(R,I)$ be a split nilpotent extension of $S$ of nilpotency degree~$N$ such that~${H^{n-1}_{dR}(R,I)=0}$ for a natural number $n\geqslant 1$. Let $J\subset R$ be an ideal contained in $I$ such that
$H^{n-1}_{dR}(R',I')\ne 0$, where $R'=R/J$, $I'=I/J$. Then there does not exist a Bloch map from~$K_{n+1}^M(R,J)$ to~${\Omega^n_{R,J}/d\,\Omega^{n-1}_{R,J}}$.
\end{proposition}
\begin{proof}
Assume the converse, that is, that the Bloch map ${\B_{R,J}\colon K_{n+1}^M(R,J)\to\Omega^n_{R,J}/d\,\Omega^{n-1}_{R,J}}$ does exist. Both $(R,I)$ and $(R',I')$ satisfy the conditions of Theorem~\ref{the:Bloch} and Theorem~\ref{thm:main}. Therefore the Bloch maps~$\B_{R,I}$ and~$\B_{R',I'}$ exist and are isomorphisms. Formula~\eqref{eq:defBloch} from Definition~\ref{def:Blochmap} implies that the Bloch map is functorial with respect to a ring and an ideal. Hence the exact sequence~\eqref{eq:exactMilnor} from Subsection~\ref{subsec:Milnorbasic} of relative Milnor $K$-groups together with the exact sequence~\eqref{eq:longseqforms} from Subsection~\ref{subsec:basicdiffforms} give a commutative diagram
$$
\begin{CD}
0@>>> K^M_{n+1}(R,J)@>>>  K^M_{n+1}(R,I)@>>>  K^M_{n+1}(R',I')@>>>0 \\
@.  @VV{\B}_{R,J}V @VV{\B}_{R,I}V @VV{\B}_{R',I'}V\\
H^{n-1}_{dR}(R',I')@>\alpha>>\Omega^n_{R,J}/d\,\Omega^{n-1}_{R,J} @>>> \Omega^n_{R,I}/d\,\Omega^{n-1}_{R,I} @>>> \Omega^n_{R',I'}/d\,\Omega^{n-1}_{R',I'}@>>>0
\end{CD}
$$
with exact rows and with $\B_{R,I}$ and $\B_{R',I'}$ being isomorphisms. This implies the vanishing
\begin{equation}\label{eq:zero}
{\rm Im}(\B_{R,J})\cap {\rm Im}(\alpha)=0
\end{equation}
of the intersection of subgroups in $\Omega^n_{R,J}/d\,\Omega^{n-1}_{R,J}$.

On the other hand, the map~$\alpha$ is injective, because of the condition $H^{n-1}_{dR}(R,I)=0$ and the exact sequence~\eqref{eq:longseqforms} extended to the left. Combining this with the condition $H^{n-1}_{dR}(R',I')\ne 0$, we obtain that ${\rm Im}(\alpha)\ne 0$. At the same time, by Lemma~\ref{lem:surjBlochmap}, the map $\B_{R,J}$ is surjective. Thus we obtain a contradiction with formula~\eqref{eq:zero}.
\end{proof}

There are many examples that satisfy the conditions of Proposition~\ref{prop:nonBloch} (see the discussion at the end of Subsection~\ref{subsec:reldR}). For instance, it follows from Proposition~\ref{prop:trivdR} and~\cite[\S\,1.3]{GrauertKerner} that one can take $S=\Q$, $N=6$, $n=1$,
$$
R=\Q[t_1,t_2]/(t_1,t_2)^6\,,\qquad I=(t_1,t_2)\,,\qquad J=(\partial_{t_1}f,\partial_{t_2}f)/(t_1,t_2)^6\,,
$$
where ${f=t_1^4+t_1^2 t_2^3+t_2^5\in \Q[t_1,t_2]}$.

\bibliographystyle{alpha}
\bibliography{RelMil}

\end{document}